\documentclass[reqno]{amsart}
\usepackage{amsfonts}
\usepackage{amssymb}
\usepackage{amsmath}
\usepackage{amsthm}
\usepackage{amscd}
\usepackage{mathrsfs}
\usepackage{graphicx, color}
\usepackage{url}
\usepackage[stable]{footmisc}

\theoremstyle{plain}
\newtheorem{theorem}{Theorem}[section]
\newtheorem{lemma}[theorem]{Lemma}
\newtheorem{corollary}[theorem]{Corollary}
\newtheorem{proposition}[theorem]{Proposition}

\theoremstyle{definition}

\newtheorem{remark}[theorem]{Remark}

\allowdisplaybreaks[1]

\newcommand{\RotE}{\mathrm{Rot} \hspace{0.1em} \mathbb{E}^{2}}

\begin{document}

\title[R-equivalence classes of $\RotE$-colorings of torus knots]{R-equivalence classes of $\RotE$-colorings \\ of torus knots}
\author{Mai Sato}
\address{Department of Mathematics, Tsuda University, 2-1-1 Tsuda-machi, Kodaira-shi, Tokyo 187-8577, Japan}
\email{md2461sm@gm.tsuda.ac.jp}


\begin{abstract}
We introduce a new equivalence relation, named R-equivalence relation, on the set of colorings of an oriented knot diagram by a quandle.
We determine the R-equivalence classes of colorings of a diagram of a torus knot by a quandle, called $\RotE$, under a certain condition.
\end{abstract}

\maketitle

\section{Introduction}
\label{sec:introduction}

A quandle, introduced by Joyce \cite{J1982}, is an algebraic system whose axioms have close relationships with Reidemeister moves for oriented knot diagrams.
Although it is called a distributive groupoid instead of a quandle, the same notion was also introduced by Matveev \cite{M1984}.
For each quandle $X$, we may consider $X$-colorings of an oriented knot diagram.
It is well-known that a Reidemeister move naturally relates the $X$-colorings of the original diagram to the ones of the diagram obtained from the original one by the Reidemeister move, one-to-one onto.
Therefore, the number of $X$-colorings gives us an invariant of oriented knots.
We call this number the $X$-coloring number.

$X$-coloring numbers are enhanced by Cater et al.\ \cite{CJKLS2003} as follows.
They defined (co)homology groups of a quandle and an weight of each $X$-coloring associated with a $2$-cocycle of $X$.
They showed that the weights of $X$-colorings, which are related to each other by a finite sequence of Reidemeister moves, associated with a $2$-cocycle are the same.
Therefore, associated with each $2$-cocycle of $X$, the multiset consisting of weights of the $X$-colorings also gives us an invariant of oriented knots.
We call it a quandle cocycle invariant.

We note that we may deform an oriented knot diagram to itself by a finite sequence of Reidemeister moves.
On the other hand, the $X$-colorings of the diagram related to each other by the sequence are not always the same.
Therefore, we say that $X$-colorings of a given diagram to be R-equivalent if they are related to each other by a finite sequence of Reidemeister moves in this paper.
Obviously, R-equivalence is an equivalence relation on the set of $X$-colorings of a diagram.
Since the weights of $X$-colorings being R-equivalent to each other are the same, determining the R-equivalence classes is allowed to be considered as an enhancement of quandle cocycle invariants.

There are several studies on relationship between $X$-colorings of a given diagram.
For example, Ge et al.\ introduced the notion of an equivalence relation on the set of $m$-colorings ($m > 1$), which are a type of $X$-coloring, of a link diagram \cite{G2014}.
Cho and Nelson introduced the notion of a quandle coloring quiver on the set of $X$-colorings of a link diagram \cite{N2019}.
We note that R-equivalence is quite different from those notions.

In this paper, we mainly focus on the following quandle to study R-equivalence.
As mentioned in Section \ref{sec:equivalence_of_coloring}, the set consisting of the rotational transformations of the Euclidean plane $\mathbb{E}^{2}$ is equipped with a quandle structure.
We let $\RotE$ denote this quandle.
Inoue studied $\RotE$-colorings of a diagram of a torus knot \cite{I2015}.
In light of his work, we determine the R-equivalence classes of $\RotE$-colorings of a diagram of a torus knot under a certain condition (Theorem \ref{main_thm}).

This paper is organized as follows.
In Section \ref{sec:equivalence_of_coloring}, we quickly review notions of a quandle and a coloring, and introduce the notion of R-equivalence.
In Section \ref{sec:RotE-coloring_of_torus_knots}, we recall Inoue's work \cite{I2015} on $\RotE$-colorings of a diagram of a torus knot.
In Section \ref{sec:main_theorem}, we introduce two significant deformations of the diagram and state the main claim (Theorem \ref{main_thm}) of this paper.
To prove Theorem \ref{main_thm}, we see two necessary conditions for $\RotE$-colorings of the diagram to be R-equivalent (Theorems \ref{necessary_condition_1} and \ref{necessary_condition_2}) in Section \ref{sec:necessary_conditions_for_R-equivalence}.
Then, in Section \ref{sec:proof_of_main_theorem}, we prove Theorem \ref{main_thm} in light of those necessary conditions.

\section{R-equivalence for quandle colorings}
\label{sec:equivalence_of_coloring}

In this section, we will introduce an equivalence relation, named R-equivalence relation, on the set of quandle colorings of an oriented knot diagram.
To do it, we start with reviewing the definitions of a quandle and a quandle coloring briefly.
More details can be found in \cite{K2017}, for example.

A \emph{quandle} is a set $X$ equipped with a binary operation $\ast : X \times X \to X$ satisfying the following three axioms:
\begin{itemize}
 \item[Q1.] For any $x \in X$, $x \ast x = x$.
 \item[Q2.] There exists a binary operation ${\ast}^{-1} : X \times X \to X$ satisfying $\left( x \ast y \right) \, {\ast}^{-1} \, y = \left( x \, {\ast}^{-1} \, y \right) \ast y = x$ for any $x, y \in X$.
 \item[Q3.] For any $x ,y, z \in X$, $\left( x \ast y \right) \ast z = \left( x \ast z \right) \ast \left( y \ast z \right)$.
\end{itemize}
For example, the cyclic group $\mathbb{Z} / n \mathbb{Z}$ equipped with a binary operation $\ast$ given by 
\[
x \ast y = 2y - x
\]
becomes a quandle for each positive integer $n \geq 3$.
We call it the \emph{dihedral quandle} of order $n$.

In this paper, we are mainly interested in the following quandle.
Define a binary operation $\ast$ on $\mathbb{C} \times \mathrm{U}(1)$ by
\[
 \left( z, e^{\theta \sqrt{-1}} \right) \ast \left( w, e^{\eta \sqrt{-1}} \right)
 = \left( \left( z - w \right)e^{\eta \sqrt{-1}} + w, e^{\theta \sqrt{-1}} \right).
\]
Then it is routine to check that $\ast$ satisfies the axioms of a quandle with a binary operation ${\ast}^{-1}$ on $\mathbb{C} \times \mathrm{U}(1)$ given by
\[
 \left( z, e^{\theta \sqrt{-1}} \right) \, {\ast}^{-1} \, \left( w, e^{\eta \sqrt{-1}} \right)
 =\left( \left( z - w \right)e^{- \eta \sqrt{-1}} + w, e^{\theta \sqrt{-1}} \right).
\]
We thus have a quandle $\left( \mathbb{C} \times \mathrm{U}(1), \ast \right)$.
We note that the rotational transformation about $w$ by angle $\eta$ maps $z$ to $\left( z - w \right)e^{\eta \sqrt{-1}} + w$.
Therefore, we can regard this quandle as a quandle consisting of the rotational transformations of $\mathbb{C}$, identifying $(z, e^{\theta \sqrt{-1}}) \in \mathbb{C} \times \mathrm{U}(1)$ with the rotational transformation about $z$ by angle $\theta$.
Furthermore, identifying $\mathbb{C}$ with the Euclidean plane $\mathbb{E}^{2}$, we refer to $\mathbb{C} \times \mathrm{U}(1)$ as $\RotE$ in the remaining.

Let $X$ be a quandle and $D$ an oriented knot diagram.
An \emph{$X$-coloring} of $D$ is a map $\mathscr{C}$ from the set of all arcs of $D$ to $X$ satisfying the condition depicted in Figure \ref{coloring_condition} at each crossing of $D$.
Here, $x$, $y$ and $x \ast y$ denote elements of $X$ assigned to correspondent arcs by $\mathscr{C}$ and we call them the \emph{colors} of the arcs.
We note that a constant map from the set of all arcs of $D$ to $X$ obviously satisfies the condition for an $X$-coloring.
We thus call it a \emph{trivial $X$-coloring} of $D$.
The colorings of a diagram of the trefoil knot by the dihedral quandle of order $3$ are depicted in Figure \ref{fox_coloring}, for example.
We note that $\mathscr{C}_{0}$, $\mathscr{C}_{1}$ and $\mathscr{C}_{2}$ are the trivial colorings in the figure.
\begin{figure}[htbp]
 \centering
 \includegraphics[scale=0.5]{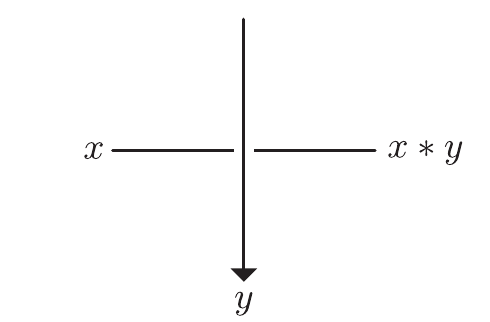}
 \caption{The condition for an $X$-coloring.}
 \label{coloring_condition}
\end{figure}
\begin{figure}[htbp]
 \centering
 \includegraphics[scale=0.5]{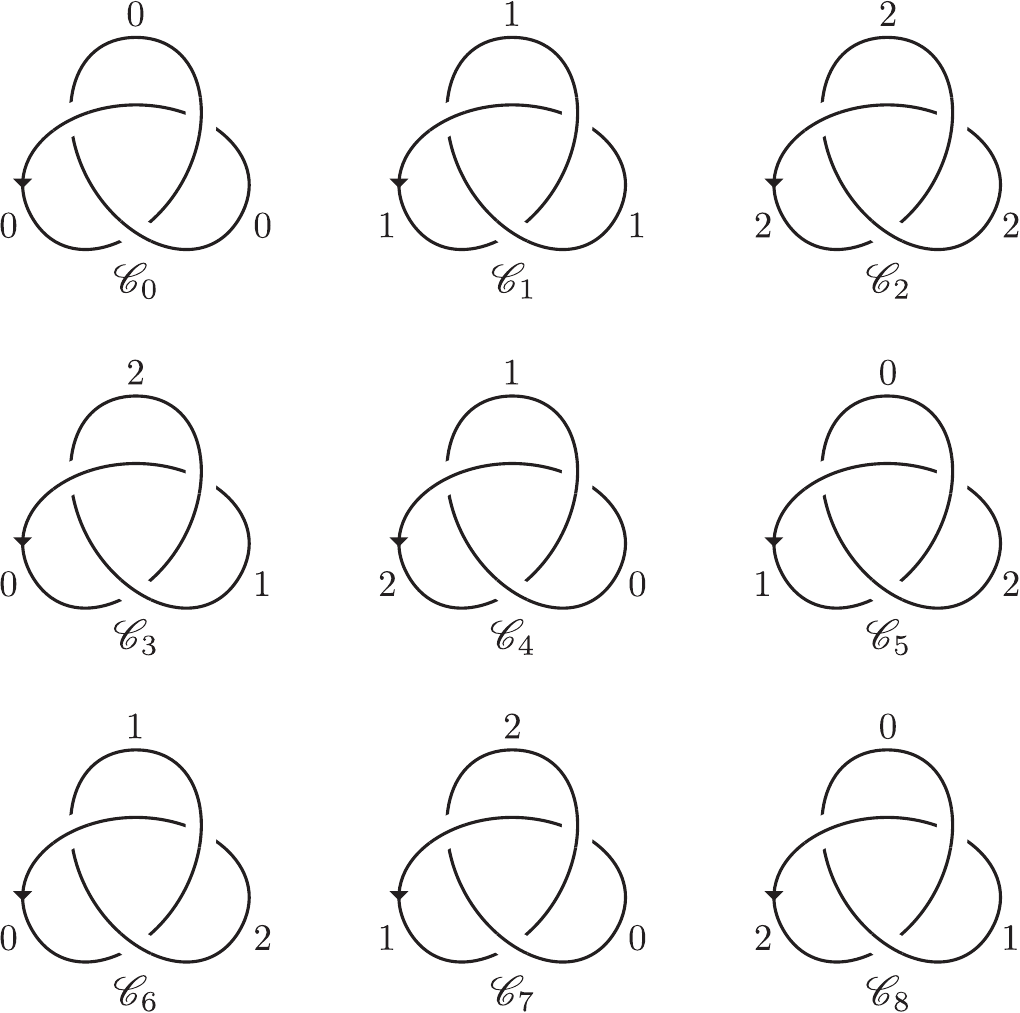}
 \caption{The colorings of a diagram of the trefoil knot by the dihedral quandle of order $3$.}
 \label{fox_coloring}
\end{figure}

It is well-known that, for each quandle $X$, a Reidemeister move yields a one-to-one correspondence between the set of $X$-colorings of a diagram $D$ and that of a diagram $D^{\prime}$ obtained from $D$ by the Reidemeister move.
Indeed, for each $X$-coloring $\mathscr{C}$ of $D$, we have a unique $X$-coloring $\mathscr{C}^{\prime}$ of $D^{\prime}$ which assigns the same colors with $\mathscr{C}$ for the arcs unrelated to the deformation and consistent colors for the others as depicted in Figure \ref{colored_Reidemeister_moves}.
We note that the axioms of a quandle guarantee the existence and uniqueness of $\mathscr{C}^{\prime}$.
A planar isotopy also yields a one-to-one correspondence between the set of $X$-colorings of a diagram $D$ and that of a diagram $D^{\prime}$ obtained from $D$ by the planar isotopy, assigning the same colors to the correspondent arcs.
In conclusion, if a diagram $D^{\prime}$ is obtained from a diagram $D$ by a finite sequence of Reidemeister moves and planar isotopies, we have a unique $X$-coloring $\mathscr{C}^{\prime}$ of $D^{\prime}$ corresponding to a $X$-coloring $\mathscr{C}$ of $D$.
We say in this situation that $\mathscr{C}^{\prime}$ is obtained from $\mathscr{C}$ by the sequence.
\begin{figure}[htbp]
 \centering
 \includegraphics[scale=0.5]{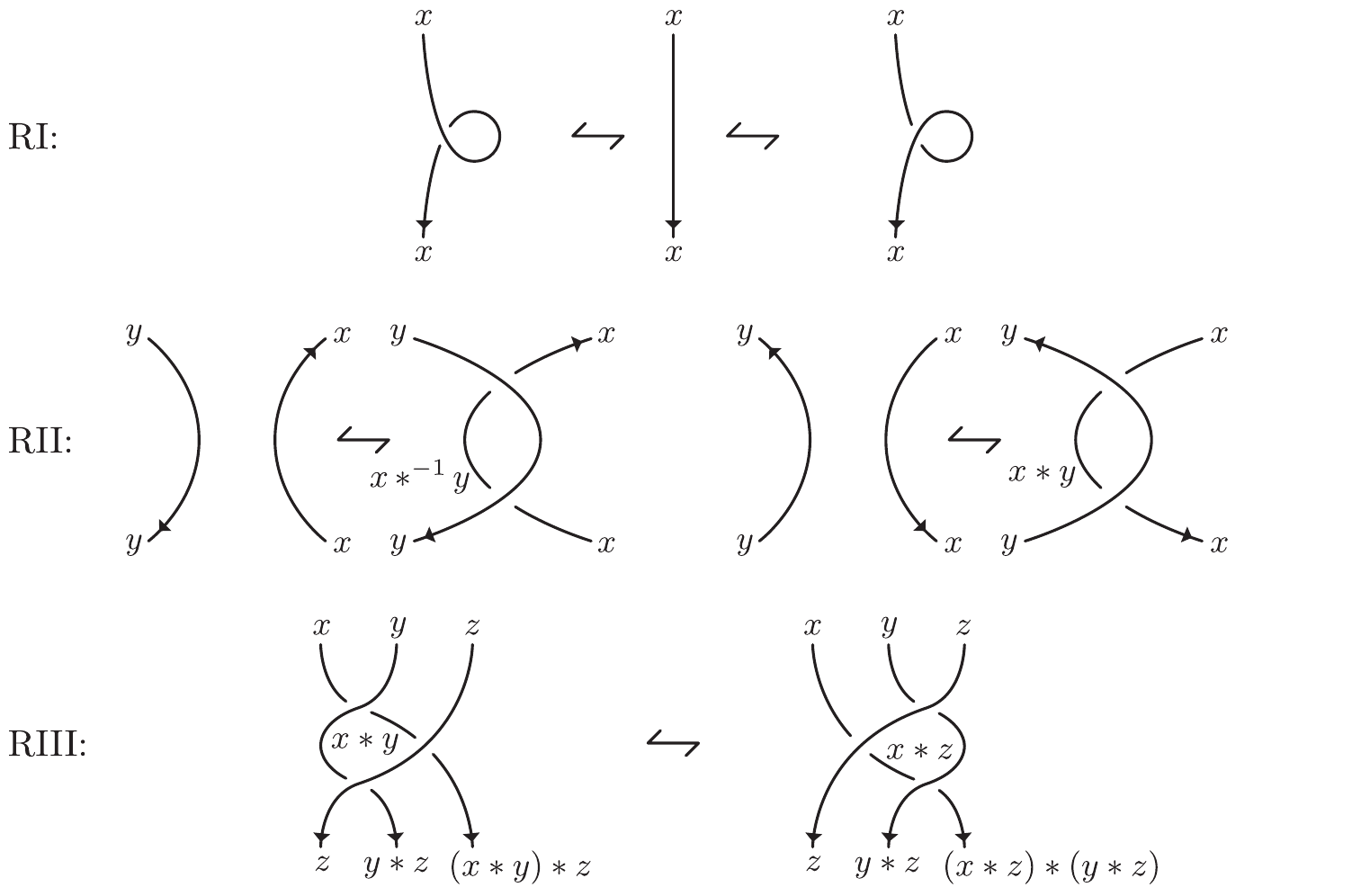}
 \caption{Reidemeister moves relate $X$-colorings of the diagrams uniquely.}
 \label{colored_Reidemeister_moves}
\end{figure}

Let $\mathscr{C}_{1}$ and $\mathscr{C}_{2}$ be $X$-colorings of a diagram.
We say that $\mathscr{C}_{2}$ is R-equivalent to $\mathscr{C}_{1}$, and write as $\mathscr{C}_{2} \sim \mathscr{C}_{1}$, if $\mathscr{C}_{2}$ is obtained from $\mathscr{C}_{1}$ by a finite sequence of Reidemeister moves and planar isotopies.
Obviously, R-equivalence yields an equivalence relation on the set of $X$-colorings of the diagram.
It is easy to see that, for each quandle $X$ and diagram $D$, a trivial $X$-coloring of $D$ is not R-equivalent to any other $X$-colorings of $D$.
Let us consider which colorings depicted in Figure \ref{fox_coloring} are R-equivalent to each other, for example.
As stated above, we know that $\mathscr{C}_{0}$, $\mathscr{C}_{1}$ and $\mathscr{C}_{2}$ are respectively not R-equivalent to any other colorings.
Since $\mathscr{C}_{4}$ and $\mathscr{C}_{5}$ are respectively obtained from $\mathscr{C}_{3}$ by the $\frac{2 \pi}{3}$- and $\frac{4 \pi}{3}$-rotations of the plane on which the diagram written, we have $\mathscr{C}_{3} \sim \mathscr{C}_{4} \sim \mathscr{C}_{5}$.
Similarly, we have $\mathscr{C}_{6} \sim \mathscr{C}_{7} \sim \mathscr{C}_{8}$.
Furthermore, since we obtain $\mathscr{C}_{7}$ from $\mathscr{C}_{3}$ by applying a switch, a shift and a switch in this order, in the sense of Section \ref{sec:main_theorem}, we have $\mathscr{C}_{3} \sim \mathscr{C}_{7}$.
Therefore, all non-trivial colorings depicted in Figure \ref{fox_coloring} are R-equivalent to each other.

\begin{remark}
 \label{remark2_for_equivalence_class}
 For some quandle $X$, let $\mathscr{C}_{1}$ and $\mathscr{C}_{2}$ be $X$-colorings of a diagram $D$.
 Furthermore, let $\mathscr{C}_{1}^{\prime}$ and $\mathscr{C}_{2}^{\prime}$ respectively be $X$-colorings of $D^{\prime}$ obtained from $\mathscr{C}_{1}$ and $\mathscr{C}_{2}$ by a finite sequence of Reidemeister moves and planar isotopies.
 Then $\mathscr{C}_{1}^{\prime}$ and $\mathscr{C}_{2}^{\prime}$ are R-equivalent to each other if and only if $\mathscr{C}_{1}$ and $\mathscr{C}_{2}$ are.
 Therefore, if we determine the R-equivalence class of $\mathscr{C}_{1}$, we also do the R-equivalence class of $\mathscr{C}_{1}^{\prime}$.
\end{remark}

\section{$\RotE$-colorings of a diagram of a torus knot}
\label{sec:RotE-coloring_of_torus_knots}

For given coprime integers $p$ and $q$ satisfying $|p|, |q| \geq 2$, let $D(p, q)$ be the diagram of the $(p, q)$-torus knot depicted in Figure \ref{a_diagram_of_(p,q)-torus_knot}.
In the remaining of this paper, we investigate R-equivalence classes of non-trivial $\RotE$-colorings of $D(p, q)$.
Since Inoue investigated the non-trivial $\RotE$-colorings of $D(p, q)$ in \cite{I2015}, we quickly review his work in this section.
\begin{figure}[htbp]
 \centering
 \includegraphics[scale=0.5]{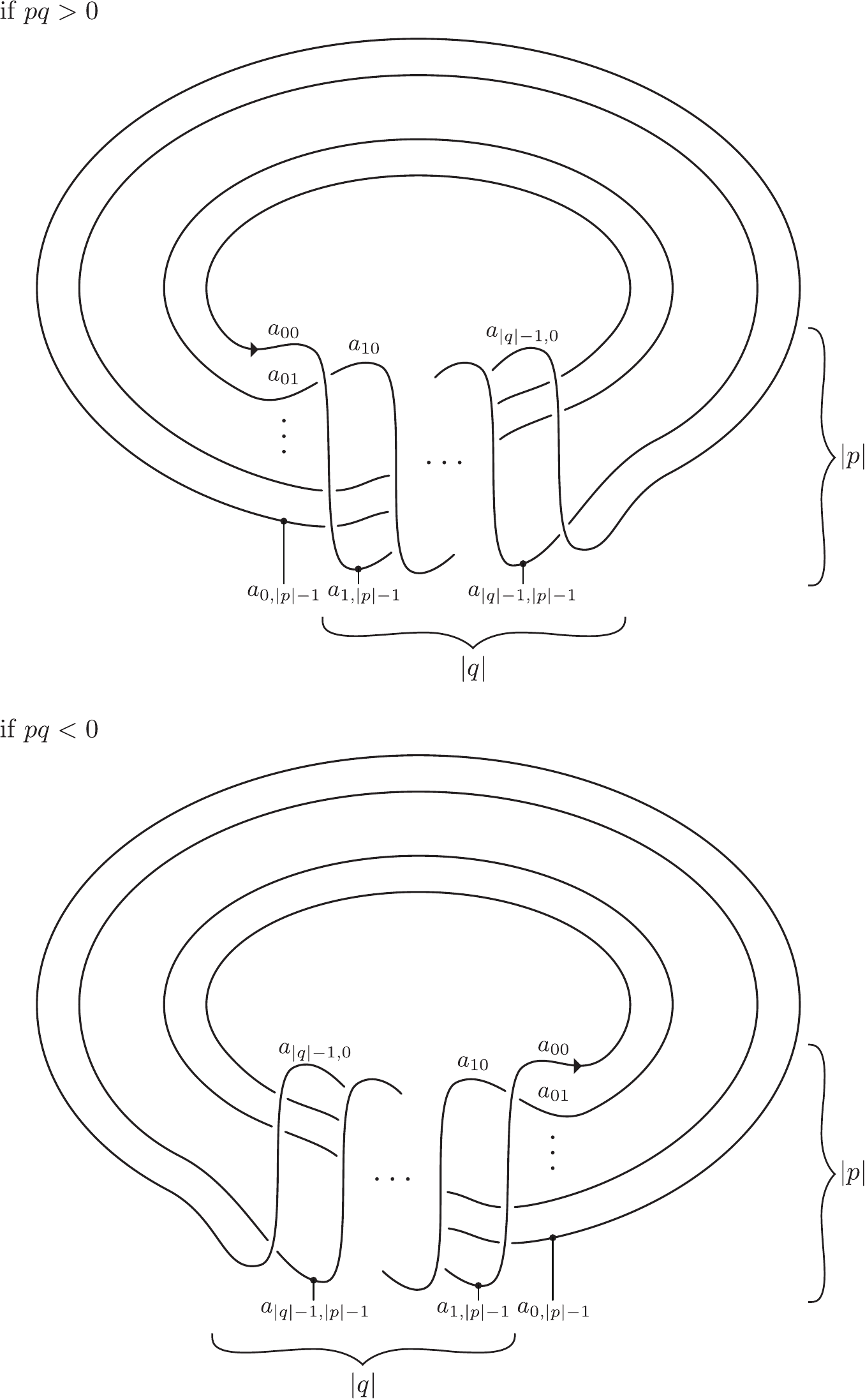}
 \caption{The diagram $D(p, q)$ of the $(p, q)$-torus knot.}
 \label{a_diagram_of_(p,q)-torus_knot}
\end{figure}

We first see that, for any $\RotE$-coloring of an oriented knot diagram $D$, the second components of the colors are the same (Lemma 3.2 of \cite{I2015}).
By the definition of $\ast$ of $\RotE$, for each crossing of $D$, the second components of the colors  of the under arcs are the same.
Since one can go around the arcs of $D$ passing under crossings, we have the claim.

We next see that we have a recipe to obtain a non-trivial $\RotE$-coloring of $D(p, q)$.
Consider a convex regular $m$-gon in $\mathbb{C}$ ($m \geq 2$) and let $v_{0}^{m}, v_{1}^{m}, \cdots , v_{m - 1}^{m}$ be the vertices of the $m$-gon in counterclockwise order.
For each integers $k$ and $i$ satisfying $1 \leq k \leq m - 1$ and $0 \leq i \leq m - 1$, we let $v_{i}^{m, k}$ denote the vertex $v_{[ik]_{m}}^{m}$.
Here, $[r]_{m}$ denotes the integer satisfying $0 \leq [r]_{m} \leq m - 1$ and $[r]_{m} \equiv r \bmod m$ for each $r \in \mathbb{Z}$.
Connect $v_{[i]_{m}}^{m, k}$ and $v_{[i + 1]_{m}}^{m, k}$ by a line segment for each $i$.
We then obtain a regular polygon.
We call it the \emph{regular polygon of type $(m, k)$} with vertices $(v_{0}^{m, k}, v_{1}^{m, k}, \cdots, v_{m-1}^{m, k})$ in this paper.
We remark that the regular polygons with vertices $(v_{0}^{m, k}, v_{1}^{m, k}, \cdots, v_{m - 1}^{m, k})$ and with vertices $(v_{[i]_{m}}^{m, k}, v_{[i + 1]_{m}}^{m, k}, \cdots, v_{[i + m - 1]_{m}}^{m, k})$ are the same for each $i \in \mathbb{Z}$.
We note that $v_{i}^{m, k} = v_{j}^{m, k}$ if $i \equiv j \bmod m^{\prime} = \frac{m}{\gcd (m, k)}$.
Thus, some vertices and edges of the polygon are overlapped if $\gcd (m, k) \neq 1$.
We furthermore note that a regular polygon of type $(m, m - k)$ is the mirror image of the regular polygon of type $(m, k)$.
We depict the cases of $m = 4, 5, 6$ in Figure \ref{regular_polygons}, for example.
It is easy to see that each side of a regular polygon of type $(m, k)$ has the same length and 
\[
 \angle v_{[i - 1]_{m}}^{m, k} v_{[i]_{m}}^{m, k} v_{[i + 1]_{m}}^{m, k} = \left| \frac{(m - 2k) \pi}{m} \right|
\]
for each $i$.
\begin{figure}[htbp]
 \centering
 \includegraphics[scale=0.5]{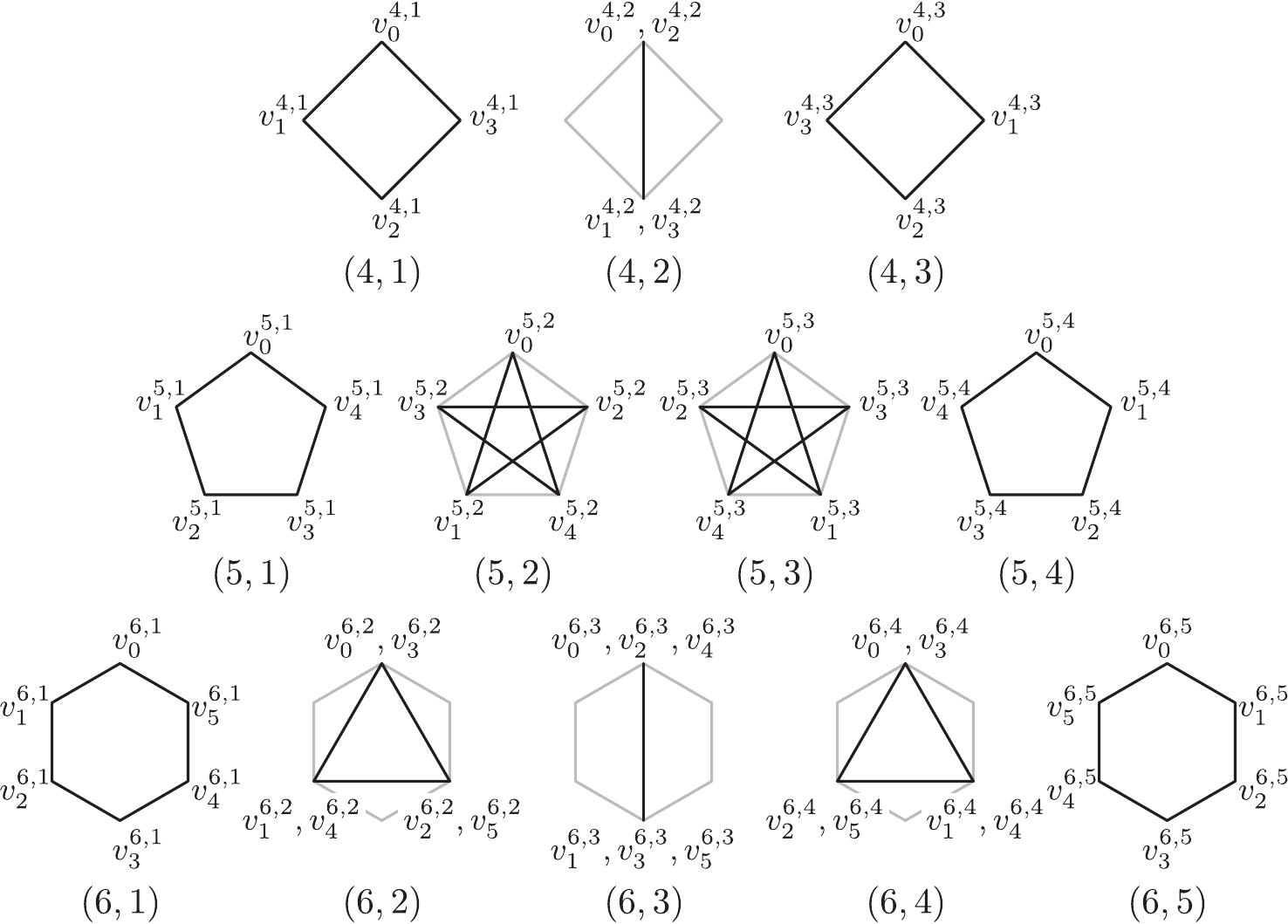}
 \caption{Regular polygons of type $(4, k)$, $(5, k)$ and $(6, k)$.}
 \label{regular_polygons}
\end{figure}

Let $Q$ be a regular polygon of type $(n, l)$ with vertices $(v_{0}, v_{1}, \cdots, v_{n - 1})$ and $P_{0}$ of type $(m, k)$ with vertices $(w_{00}, w_{01}, \cdots, w_{0, m - 1})$ ($m, n \geq 2, \ 1 \leq k \leq m - 1, \ 1 \leq l \leq n - 1$).
We assume that $w_{00} = v_{0}$ and $w_{01} = v_{1}$.
We first rotate $P_{0}$ about $w_{01} = v_{1}$ by angle 
\[
 \theta=\theta (m, k; n, l) = \left( \frac{m - 2k}{m}-\frac{n - 2l}{n} \right) \pi.
\]
Then we obtain the regular polygon $P_{1}$ of type $(m, k)$.
Let $(w_{10}, w_{11}, \cdots, w_{1, m - 1})$ be the vertices of $P_{1}$.
We may assume that $w_{11} = v_{1}$ and $w_{1, [2]_{m}} = v_{[2]_{n}}$.
We next rotate $P_{1}$ about $w_{1, [2]_{m}} = v_{[2]_{n}}$ by angle $\theta$.
Then we obtain the regular polygon $P_{2}$ of type $(m, k)$ with vertices $(w_{20}, w_{21}, \cdots, w_{2, m - 1})$. 
We may assume that $w_{2, [2]_{m}} = v_{[2]_{n}}$ and $w_{2, [3]_{m}} = v_{[3]_{n}}$.
We continue this process in the same way: the $i$-th step is the rotation of $P_{i - 1}$ about $w_{i - 1, [i]_{m}} = v_{[i]_{n}}$ by angle $\theta$ to obtain the regular polygon $P_{i}$ of type $(m, k)$ with vertices $(w_{i0}, w_{i1}, \cdots, w_{i, m - 1})$ satisfying $w_{i, [i]_{m}} = v_{[i]_{n}}$ and $w_{i, [i + 1]_{m}} = v_{[i + 1]_{n}}$. 
We call this procedure the \emph{$(m, k; n, l)$-trochoid} around $Q$ starting at $P_{0}$.
Figure \ref{4132_trochoid} illustrates a $(4, 1; 3, 2)$-trochoid, for example.
We often draw it all together as depicted in Figure \ref{diagram_of_4132_trochoid} and call it the \emph{diagram} of the trochoid.
\begin{figure}[htbp]
 \centering
 \includegraphics[scale=0.5]{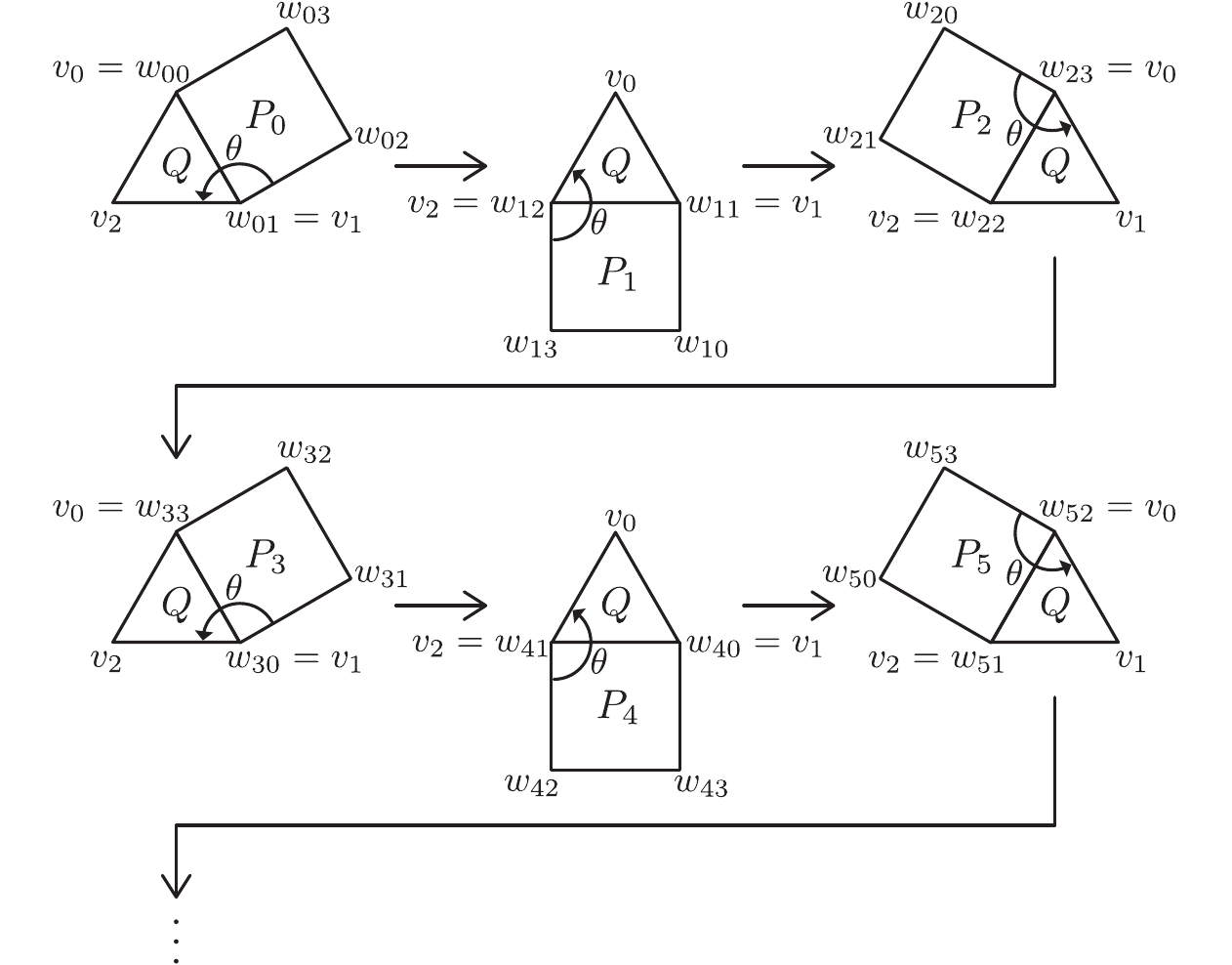}
 \caption{A $(4, 1; 3, 2)$-trochoid around $Q$ starting at $P_{0}$.}
 \label{4132_trochoid}
\end{figure}
\begin{figure}[htbp]
 \centering
 \includegraphics[scale=0.5]{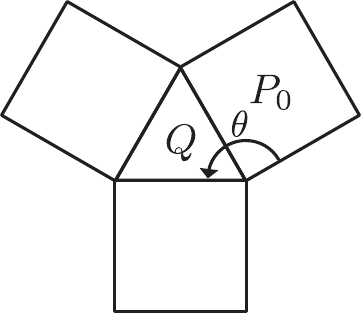}
 \caption{The diagram of the $(4, 1; 3, 2)$-trochoid depicted in Figure \ref{4132_trochoid}.}
 \label{diagram_of_4132_trochoid}
\end{figure}

\begin{theorem}[Theorem 4.1 of \cite{I2015}]
 \label{coloring_and_trochoid}
 Set $m = |p|$ and $n = |q|$ in the above.
 Let $a_{ij}$ denote the arcs of $D(p, q)$ as depicted in Figure \ref{a_diagram_of_(p,q)-torus_knot}, although $a_{i0}$ and $a_{[i + 1]_{|q|}, |p| - 1}$ mark the same arc for each $i$ \textup{(}$0 \leq i \leq |q| - 1$\textup{)}.
 Then the map $\mathscr{C}(Q, P_{0})$ from the set of all arcs of $D(p, q)$ to $\RotE$ given by 
 \[
  \mathscr{C}(Q, P_{0})(a_{ij})=\left( w_{i, [i + j + 1]_{|p|}}, e^{\theta (|p|, k; |q|, l) \sqrt{-1}} \right)
 \] 
 is well-defined and a non-trivial $\RotE$-coloring of $D(p,q)$.
 Here, we assume that names of the vertices of the polygons are given by their coordinates.
\end{theorem}

\begin{proof}
 It is routine to check that $\mathscr{C}(Q, P_{0})$ is well-defined and satisfies the condition for a $\RotE$-coloring at each crossing.
\end{proof}

We call $\mathscr{C}(Q, P_{0})$ the non-trivial $\RotE$-coloring derived from the $(|p|, k; |q|, l)$-trochoid around $Q$ starting at $P_{0}$.
Inoue also showed that any non-trivial $\RotE$-coloring of $D(p, q)$ is derived from some $(|p|, k; |q|, l)$-trochoid (see the last paragraph of p.\ 42 of \cite{I2015}).
Therefore, we have a one-to-one correspondence between the set of non-trivial $\RotE$-colorings of $D(p, q)$ and the set of $(|p|, k; |q|, l)$-trochoids ($1 \leq k \leq |p| - 1, \ 1 \leq l \leq |q| - 1$).

\begin{remark}
 \label{how_to_know_all_colors}
 For a non-trivial $\RotE$-coloring $\mathscr{C}(Q, P_{0})$ of $D(p, q)$, the vertex $v_{i} = w_{i, [i]_{|p|}}$ of $Q$ is appeared as the first component of the color of the arc $a_{i, |p| - 1}$ ($0 \leq i \leq |q| - 1$).
 Similarly, the vertex $w_{0, [j + 1]_{|p|}}$ of $P_{0}$ is appeared as the first component of the color of the arc $a_{0j}$ ($0 \leq j \leq |p| - 1$).
 Therefore, if we know the first components of the colors of the arcs $a_{i, |p| - 1}$ or $a_{0j}$ for a non-trivial $\RotE$-coloring, we can determine the trochoid from which the coloring derived.
\end{remark}

\section{Main theorem}
\label{sec:main_theorem}

The aim of this section is to claim that we can determine the R-equivalence classes of non-trivial $\RotE$-colorings of $D(p, q)$ under a certain condition.
To do it, we first introduce the following two deformations of $D(p, q)$.

The first one is a \emph{shift} which deforms $D(p, q)$ to $D(p, q)$ by the planar isotopy sending the part surrounded by the rectangle in the left-hand side of Figure \ref{pic_shift} to the one in the right-hand side.
The second one is a \emph{switch} which deforms $D(p, q)$ to $D(q, p)$ as explained in Figures \ref{dual_1}--\ref{dual_3} considering diagrams as ``pictures'' of the torus knot which lies on a once punctured torus.
Although the figures illustrate the case $pq > 0$, we have a similar deformation even if $pq < 0$.
We note that a sequence of shifts and even times of switches deforms $D(p,q)$ to $D(p,q)$.
\begin{figure}[htbp]
 \centering
 \includegraphics[scale=0.5]{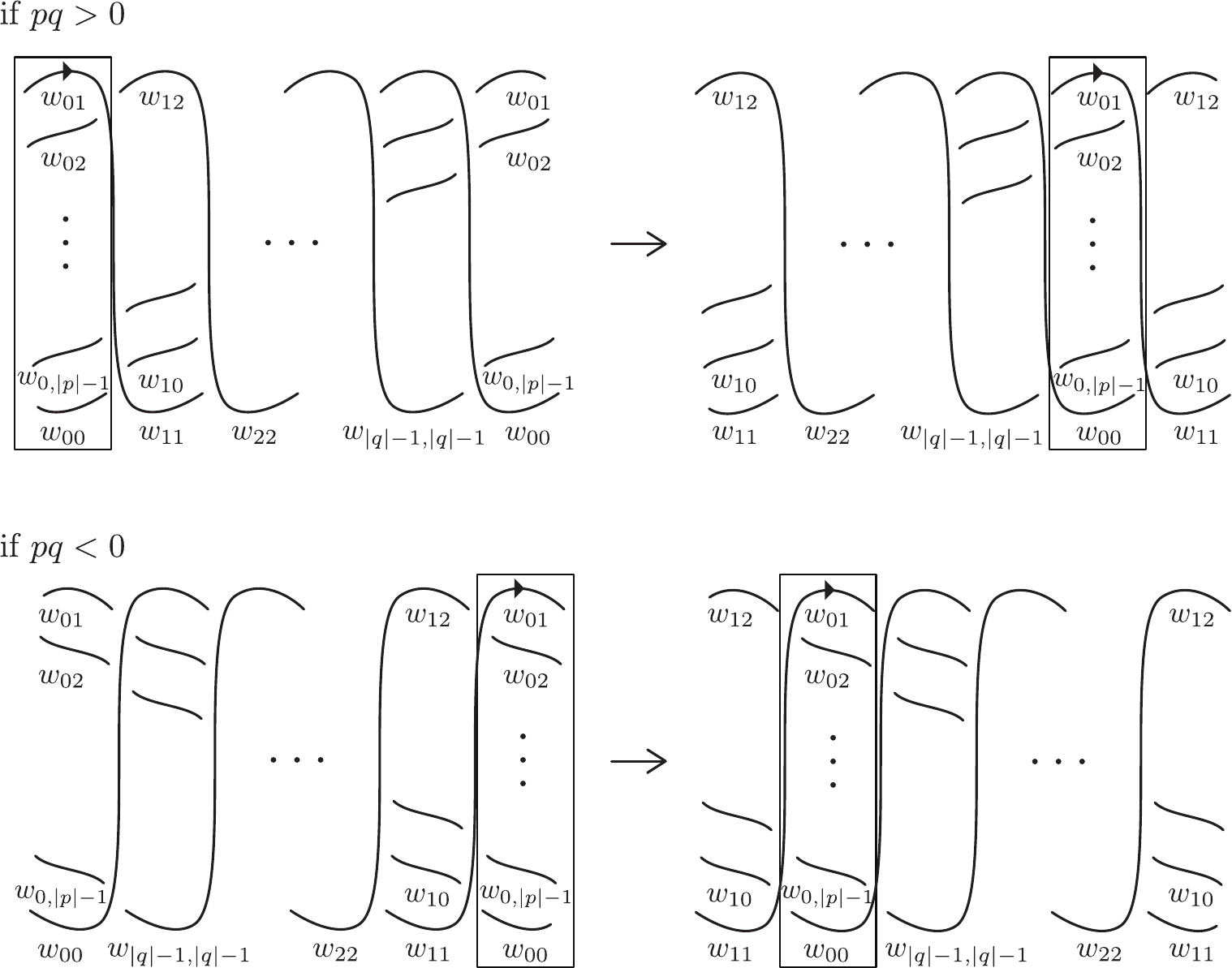}
 \caption{We obtain $\mathscr{C}(Q, P_{1})$ from $\mathscr{C}(Q, P_{0})$ by a planar isotopy.}
 \label{pic_shift}
\end{figure}
\begin{figure}[htbp]
 \centering
 \includegraphics[scale=0.5]{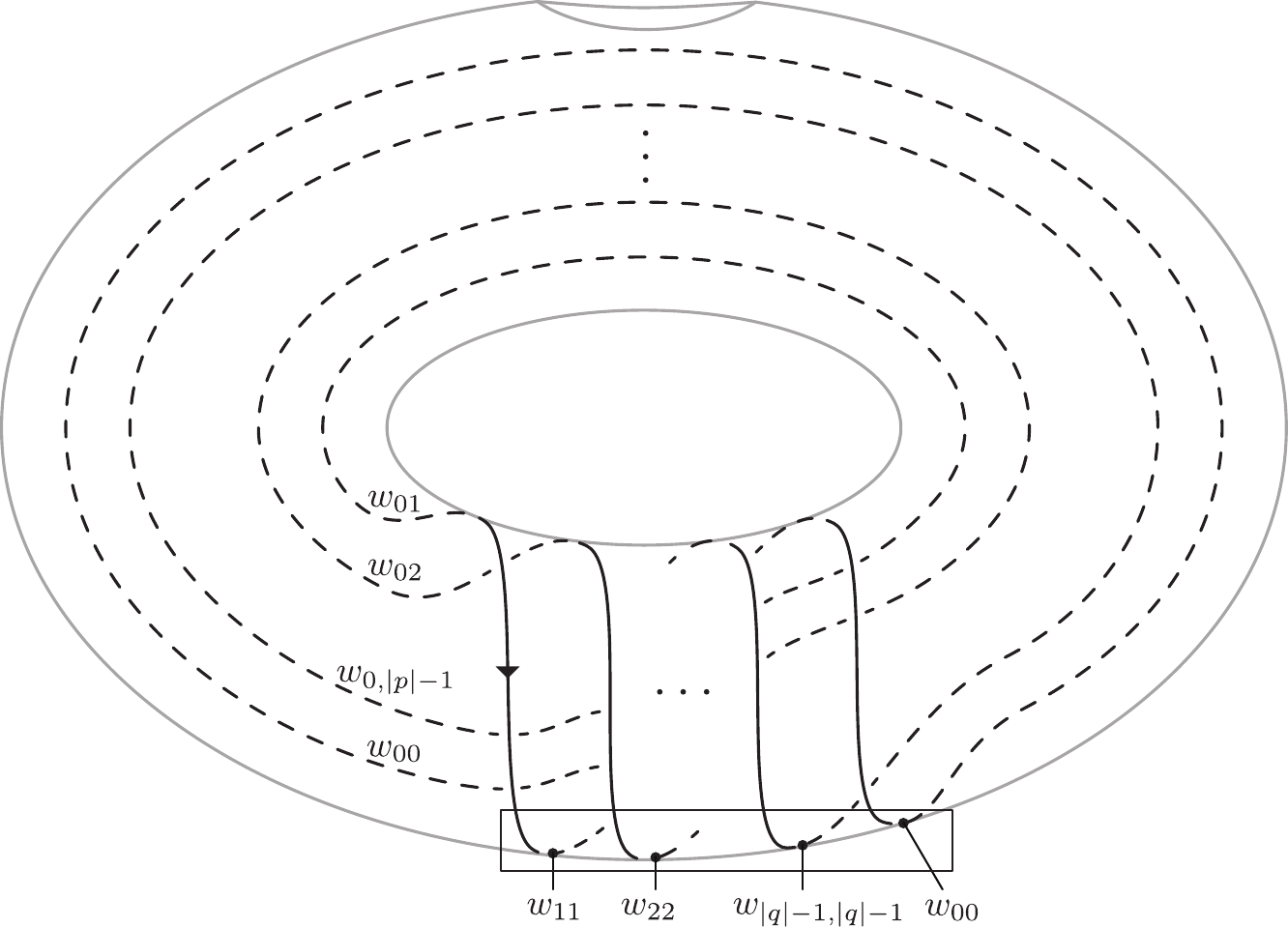}
 \caption{We lay the $(p, q)$-torus knot on a once punctured torus.}
 \label{dual_1}
\end{figure}
\begin{figure}[htbp]
 \centering
 \includegraphics[scale=0.5]{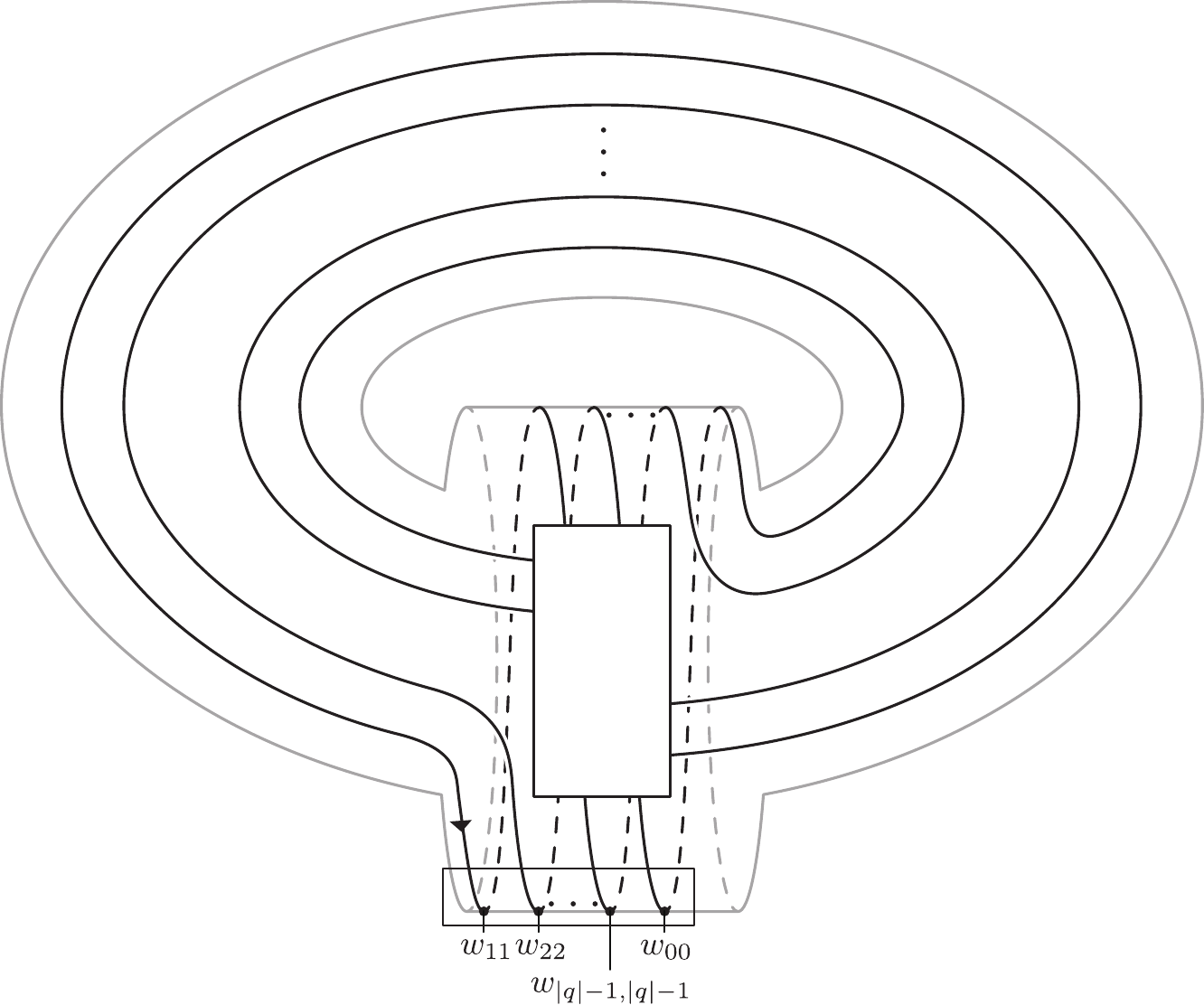}
 \caption{We widen the puncture of the torus depicted in Figure \ref{dual_1}.}
 \label{dual_2}
\end{figure}
\begin{figure}[htbp]
 \centering
 \includegraphics[scale=0.5]{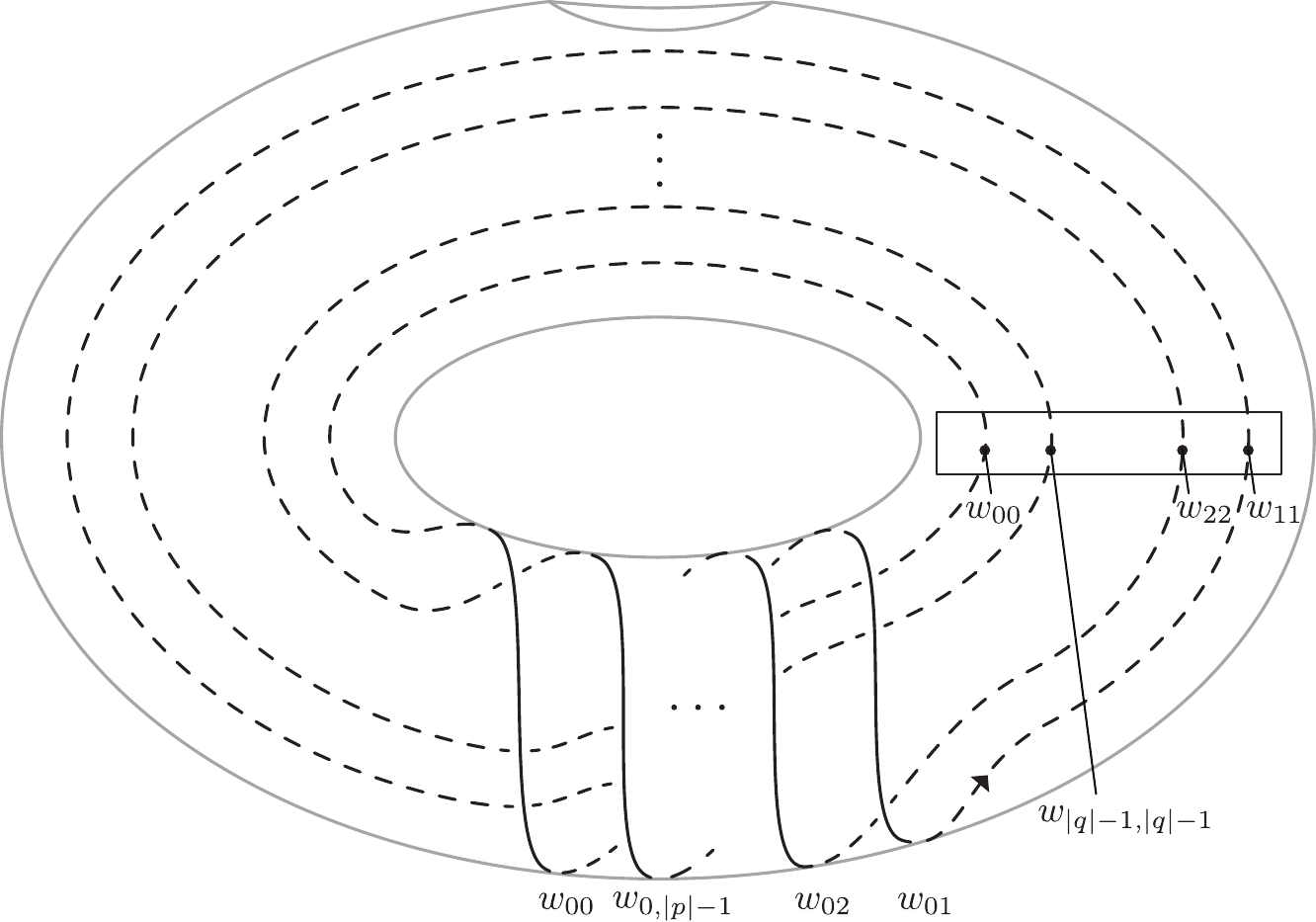}
 \caption{We narrow the puncture of the torus depicted in Figure \ref{dual_2} turning the torus inside out.}
 \label{dual_3}
\end{figure}

We next see the effect of a shift or a switch on a non-trivial $\RotE$-coloring of $D(p, q)$ and the trochoid from which the coloring is derived.
Suppose that $D(p, q)$ is equipped with a non-trivial $\RotE$-coloring $\mathscr{C}(Q, P_{0})$ derived from a $(|p|,k;|q|,l)$-trochoid as depicted in the left-hand side of Figure \ref{pic_shift} or Figure \ref{dual_1}.
Here, the second components of the colors are omitted in the figures for viewability.
Furthermore, although we should estimate the second index $j$ of each $w_{ij}$ by the function $[\, \cdot \,]_{|p|}$ strictly speaking, we omit it in the figures for simplicity.
We first consider the coloring of $D(p, q)$ obtained from $\mathscr{C}(Q, P_{0})$ by a shift, that is depicted in the right-hand side of Figure \ref{pic_shift}.
It is easy to see that this coloring is $\mathscr{C}(Q, P_{1})$, derived from the $(|p|, k; |q|, l)$-trochoid around $Q$ starting at $P_{1}$ instead of $P_{0}$.
We also call the deformation of the diagram of the trochoid from which $\mathscr{C}(Q, P_{0})$ is derived to the one from which $\mathscr{C}(Q, P_{1})$ is derived a \emph{shift}.
We next consider the coloring of $D(q, p)$ obtained from $\mathscr{C}(Q, P_{0})$ by a switch.
Since no arcs fly over the parts surrounded by rectangles in Figures \ref{dual_1}--\ref{dual_3} during a switch, the colors of the arcs in the interior of the rectangles do not change.
Thus, in light of Remark \ref{how_to_know_all_colors}, we know that the coloring of $D(q, p)$ obtained from $\mathscr{C}(Q, P_{0})$ by a switch is the one depicted in Figure \ref{dual_3}.
To describe the coloring, we let $\langle P_{0} \rangle$ denote the regular polygon of type $(|p|, |p| - k)$ with vertices $(w_{01}, w_{00}, w_{0, |p| - 1}, \cdots, w_{02})$ and $\langle Q \rangle _{0}$ of type $(|q|, |q| - l)$ with vertices $(v_{1}, v_{0}, v_{|q| - 1}, \cdots, v_{2}) = (w_{11}, w_{00}, w_{|q| - 1, [|q| - 1]_{|p|}}, \cdots, w_{2, [2]_{|p|}})$.
We note that $P_{0}$ and $\langle P_{0} \rangle$ (respectively $Q$ and $\langle Q \rangle _{0}$) have the same shape in $\mathbb{C}$.
Then it is routine to see that the coloring of $D(q, p)$ obtained from $\mathscr{C}(Q, P_{0})$ by a switch is $\mathscr{C}(\langle P_{0} \rangle, \langle Q \rangle _{0})$ derived from the $(|q|, |q| - l; |p|, |p| - k)$-trochoid around $\langle P_{0} \rangle$ starting at $\langle Q \rangle _{0}$.
We also call the deformation of the diagram of the trochoid from which $\mathscr{C}(Q, P_{0})$ is derived to the one from which $\mathscr{C}(\langle P_{0} \rangle, \langle Q \rangle _{0})$ is derived a \emph{switch}.

For non-trivial $\RotE$-colorings of $D(p, q)$, we have the following theorem.

\begin{theorem}
 \label{main_thm}
 Let $\mathscr{C}_{0}$ be a non-trivial $\RotE$-coloring of $D(p, q)$ derived from a $(|p|, k; |q|, l)$-trochoid.
 Furthermore, let $p^{\prime} = \frac{|p|}{\gcd(|p|, k)}$ and $q^{\prime} = \frac{|q|}{\gcd(|q|, l)}$.
 If $p^{\prime} q^{\prime}$ is even, then a non-trivial $\RotE$-coloring $\mathscr{C}$ of $D(p, q)$ is R-equivalent to $\mathscr{C}_{0}$ if and only if $\mathscr{C}$ is obtained from $\mathscr{C}_{0}$ by shifts and even times of switches.
\end{theorem}

It is clear that $\mathscr{C}$ is R-equivalent to $\mathscr{C}_{0}$ if $\mathscr{C}$ is obtained from $\mathscr{C}_{0}$ by shifts and even times of switches.
Therefore, we devote the remaining of this paper to prove the ``only if'' part.

\section{Two necessary conditions for R-equivalence}
\label{sec:necessary_conditions_for_R-equivalence}

In this section, we introduce two necessary conditions for non-trivial $\RotE$-colorings of $D(p, q)$ to be R-equivalent, which are stated in the following theorems.

\begin{theorem}
\label{necessary_condition_1}
 Let $\mathscr{C}$ and $\mathscr{C}^{\prime}$ respectively be non-trivial $\RotE$-colorings of $D(p, q)$ derived from a $(|p|, k; |q|, l)$-trochoid and a $(|p|, k^{\prime}; |q|, l^{\prime})$-trochoid.
 If $\mathscr{C}^{\prime}$ is R-equivalent to $\mathscr{C}$, then $k^{\prime} = k$ and $l^{\prime} = l$.
\end{theorem}

\begin{proof}
 Since the second components of the colors are preserved under Reidemeister moves and planar isotopies, we have the claim.
\end{proof}

\begin{theorem}
 \label{necessary_condition_2}
 Let $\mathscr{C}$ and $\mathscr{C}^{\prime}$ respectively be non-trivial $\RotE$-colorings of $D(p, q)$ derived from $(|p|, k; |q|, l)$-trochoids around $Q$ and $Q^{\prime}$.
 If $\mathscr{C}^{\prime}$ is R-equivalent to $\mathscr{C}$, then the side lengths of $Q$ and $Q^{\prime}$ are the same.
\end{theorem}

In the remaining of this section, we will prove the second theorem.
To do it, we will use weights of colorings associated with a quandle $2$-cocycle.
Thus, we start with reviewing it briefly.
More details can be found in \cite{K2017}, for example.

Let $X$ be a quandle and $A$ an abelian group.
A map $f : X \times X \to A$ is said to be a \emph{quandle $2$-cocycle} if $f$ satisfies the following two conditions:
\begin{itemize}
 \item[QC1.] For any $x, y, z \in X$, $f(x, y) + f(x \ast y, z) - f(x, z) - f(x \ast z, y \ast z) = 0$.
 \item[QC2.] For any $x \in X$, $f(x, x) = 0$.
\end{itemize}

Suppose that $f : X \times X \to A$ is a quandle $2$-cocycle.
Let $\mathscr{C}$ be an $X$-coloring of an oriented knot diagram $D$.
For each crossing $v$ of $D$, whose under and over arcs are colored by $\mathscr{C}$ as depicted in Figure \ref{coloring_condition}, we define a \emph{local weight} $W_{f}(v, \mathscr{C})$ of $\mathscr{C}$ at $v$ by  
\[
W_{f}(v, \mathscr{C}) = \varepsilon \cdot f(x, y),
\]
where $\varepsilon = + 1$ if $v$ is positive, otherwise $\varepsilon = - 1$.
Take the sum 
\[
W_{f}(D, \mathscr{C}) = \sum_{v} W_{f}(v, \mathscr{C})
\]
of the local weights over all crossings of $D$ and call it the \emph{weight} of $\mathscr{C}$.
It is easy to see that if an $X$-coloring $\mathscr{C}^{\prime}$ of $D^{\prime}$ is obtained from $\mathscr{C}$ by a finite sequence of Reidemeister moves and planar isotopies, then we have $W_{f}(D^{\prime}, \mathscr{C}^{\prime}) = W_{f}(D, \mathscr{C})$ (see Figure \ref{weight_Reidemeister_moves}).
Thus, weights of $X$-colorings which are R-equivalent to each other are the same.
\begin{figure}[htbp]
 \centering
 \includegraphics[scale=0.5]{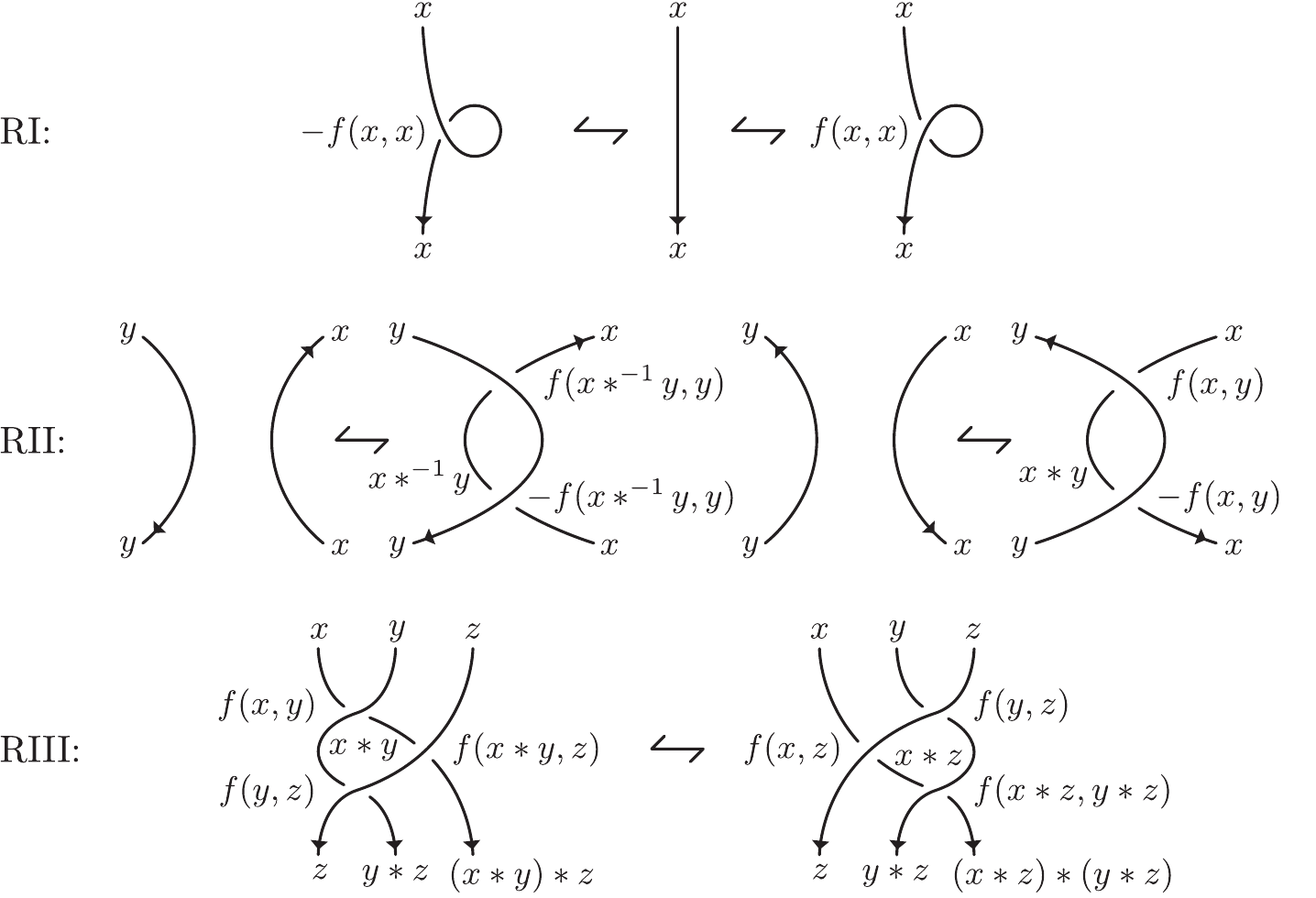}
 \caption{Reidemeister moves do not change the weights of $X$-colorings of the diagrams.}
 \label{weight_Reidemeister_moves}
\end{figure}

We next introduce a concrete quandle $2$-cocycle of $\RotE$.
Let $C_{n}(\mathbb{C})$ be the free $\mathbb{Z}$-module generated by $\mathbb{C}^{n}$ and $s : C_{3}(\mathbb{C}) \to \mathbb{R}$ a $\mathbb{Z}$-homomorphism  given by
\[
 s(x, y, z) = \delta \cdot (\text{the area of the triangle} \ (x, y, z)), 
\]
where the area of a degenerated triangle is zero, and $\delta = + 1$ if $x, y, z$ are arranged counterclockwise in this order, otherwise $\delta = - 1$.
We call $s(x, y, z)$ the \emph{signed area} of the triangle $(x, y, z)$.
We have the following two lemmas about signed area.

\begin{lemma}
 \label{sign_area}
 For each $x, y, z, w \in \mathbb{C}$ and $e^{\theta \sqrt{-1}} \in \text{U}(1)$, we have 
 \[
  s(x, y, z) = s \left( (x - w)e^{\theta \sqrt{-1}} + w, (y - w)e^{\theta \sqrt{-1}} + w, (z - w)e^{\theta \sqrt{-1}} + w \right).
 \]
\end{lemma}

\begin{proof}
 Suppose that 
 \[
  x^{\prime} = (x - w)e^{\theta \sqrt{-1}} + w, \quad y^{\prime} = (y - w)e^{\theta \sqrt{-1}} + w, \quad z^{\prime} = (z - w)e^{\theta \sqrt{-1}} + w.
 \]
 Then the triangle $(x^{\prime}, y^{\prime}, z^{\prime})$ is obtained from the triangle $(x, y, z)$ by the $\theta$-rotation about $w$.
 Since the rotation does not change the area and the orders of vertices, we have the claim.
\end{proof}

\begin{lemma}
 \label{0-map}
 Define a $\mathbb{Z}$-homomorphism $\partial : C_{4}(\mathbb{C}) \to C_{3}(\mathbb{C})$ by
 \[
  \partial(x, y, z, w) = (y, z, w) - (x, z, w) + (x, y, w) - (x, y, z).
 \]
 Then $s \circ \partial$ is the $0$-map.
\end{lemma}

\begin{proof}
 Suppose first that $x$, $y$, $z$ and $w$ are arranged on a circle counterclockwise in this order.
 Then we can check that $s \circ \partial(x, y, z, w) = 0$.
 It is routine to see that this equality is still true even if we move the points $x$, $y$, $z$ and $w$ continuously and independently.
 Thus, we have the claim.
\end{proof}

Utilizing signed area, we have a quandle $2$-cocycle of $\RotE$ as follows.

\begin{proposition}
 \label{quandle_2_cocycle}
 Choose and fix $o \in \mathbb{C}$ and define a map $\Phi_{o} : \RotE \times \RotE \to \mathbb{R}$ by
 \[
  \Phi_{o} \Bigl( \left( x, e^{\theta_{1} \sqrt{-1}} \right), \left( y, e^{\theta_{2} \sqrt{-1}} \right) \Bigl)
  = - s \left( o, x, y \right) + s \left( o, (x - y)e^{\theta_{2} \sqrt{-1}} + y, y \right).
 \]
 Then $\Phi_{o}$ is a quandle $2$-cocycle of $\RotE$.
\end{proposition}

\begin{proof}
 Obviously, $\Phi_{o}$ satisfies the condition QC2. 
 Thus, we check that $\Phi_{o}$ satisfies the condition QC1.
 Let $(x, e^{\theta_{1} \sqrt{-1}}), (y, e^{\theta_{2} \sqrt{-1}}), (z, e^{\theta_{3} \sqrt{-1}}) \in \RotE$.
 It is routine to check in light of Lemma \ref{sign_area} that we have
 \begin{align*}
  & \hspace{1.3em} \Phi_{o} 
     \Bigl( \left( x, e^{\theta_{1} \sqrt{-1}} \right), \left( y, e^{\theta_{2} \sqrt{-1}} \right) \Bigl) \\
  & \hspace{1.3em} \quad +\Phi_{o} 
     \Bigl( \left( x, e^{\theta_{1} \sqrt{-1}} \right) \ast \left( y, e^{\theta_{2} \sqrt{-1}} \right),
             \left( z, e^{\theta_{3} \sqrt{-1}} \right) \Bigl) \\
  & \hspace{1.3em} \qquad -\Phi_{o} 
     \Bigl( \left( x, e^{\theta_{1} \sqrt{-1}} \right), \left( z, e^{\theta_{3} \sqrt{-1}} \right) \Bigl) \\
  & \hspace{1.3em} \quad \qquad -\Phi_{o} 
     \Bigl( \left( x, e^{\theta_{1} \sqrt{-1}} \right) \ast \left( z, e^{\theta_{3} \sqrt{-1}} \right),
             \left( y, e^{\theta_{2} \sqrt{-1}} \right) \ast \left( z, e^{\theta_{3} \sqrt{-1}} \right) \Bigl) \\
   & = s \circ \partial (o, x, y, z) \\
   & \hspace{1.3em} \quad +s \circ \partial \left( o, \left( (x - y)e^{\theta_{2} \sqrt{-1}} + y - z \right)
                                            e^{\theta_{3} \sqrt{-1}} + z, (y - z)e^{\theta_{3} \sqrt{-1}} + z, z \right) \\
   & \hspace{1.3em} \qquad -s \circ \partial \left( o, (x - y)e^{\theta_{2} \sqrt{-1}} + y, y, z \right) \\
   & \hspace{1.3em} \quad \qquad -s \circ \partial 
                      \left( o, (x - z)e^{\theta_{3} \sqrt{-1}} + z, (y - z)e^{\theta_{3} \sqrt{-1}} + z, z \right).
 \end{align*}
 By Lemma \ref{0-map}, the right-hand side of the above equality is equal to zero.
\end{proof}

We now prove Theorem \ref{necessary_condition_2} utilizing the above quandle 2-cocycle.
Let \linebreak $P = (v_{0}, \cdots, v_{m - 1})$ denote a polygon in $\mathbb{C}$ whose vertices are $v_{0}, \cdots, v_{m - 1}$ in this order.
We define the \emph{signed area} of $P$ by
\[
  S(P) = S(v_{0}, \cdots, v_{m - 1}) = \textstyle\sum\limits_{i = 0}^{m - 1} s(o, v_{i}, v_{[i + 1]_{m}})
\]
with some $o \in \mathbb{C}$.
It is routine to see that the value of $S(P)$ does not depend on the choice of $o$.
We note that $S(v_{0}, v_{1}, v_{2}) = s(v_{0}, v_{1}, v_{2})$ for a triangle $(v_{0}, v_{1}, v_{2})$.
We have the following key proposition.

\begin{proposition}
\label{weight_of_Dpq}
 Let $\mathscr{C}$ be a non-trivial $\RotE$-coloring of $D(p, q)$ derived from a $(|p|, k; |q|, l)$-trochoid around $Q$ starting at $P_{0}$ and $o$ the center of (the convex hull of the vertices of) $Q$.
 Then we have 
 \[
  W_{\Phi_{o}} \left( D(p, q),\mathscr{C} \right) =
  \varepsilon \left( S(P_{0}) \cdot |q| -S(Q) \cdot |p| \right),
 \]
 where $\varepsilon = + 1$ if $pq > 0$, otherwise $\varepsilon = - 1$.
\end{proposition}

\begin{proof}
 Let $p^{\prime} = \frac{|p|}{\gcd(|p|, k)}$, $k^{\prime} = \frac{k}{\gcd(|p|, k)}$ and $\theta = \theta(|p|, k; |q|, l)$.
 Suppose that $\mathscr{C}$ is given by 
 \[
  \mathscr{C}(a_{ij}) = \left( w_{i, [i + j + 1]_{|p|}}, e^{\theta \sqrt{-1}} \right),
 \]
 as ever.
 Then, in light of the axiom QC2 of a quandle 2-cocycle, we have 
 \begin{align}
  & \hspace{1.3em} W_{\Phi_{o}} \left( D(p, q), \mathscr{C} \right) \notag \\[0.5ex]
  & = \varepsilon \left( 
                   \ \sum_{\substack{0 \leq i \leq |q| - 1, \\ 0 \leq j \leq |p| - 1}}
                   \Phi_{o} \left( \left( w_{ij}, e^{\theta \sqrt{-1}} \right), 
                                   \left( w_{i, [i + 1]_{|p|}}, e^{\theta \sqrt{-1}} \right) 
                            \right) \right.
                    \notag \\
  & \hspace{1.3em} \qquad \left. - \sum_{\substack{0 \leq i \leq |q| - 1 \\ \phantom{|}}}
                     \Phi_{o} \left( \left( w_{i, [i + 1]_{|p|}}, e^{\theta \sqrt{-1}} \right),
                                     \left( w_{i, [i + 1]_{|p|}}, e^{\theta \sqrt{-1}} \right) \right) 
                  \right)
                  \notag \\[0.5ex]
  & = \varepsilon \ \sum_{\substack{0 \leq i \leq |q| - 1, \\ 0 \leq j \leq |p| - 1}}
                   \Phi_{o} \left( \left( w_{ij}, e^{\theta \sqrt{-1}} \right), 
                                   \left( w_{i, [i + 1]_{|p|}}, e^{\theta \sqrt{-1}} \right) 
                            \right).
                    \label{eq:A}
 \end{align}
 Furthermore, remarking that $w_{i j_{1}} = w_{i j_{2}}$ if $j_{1} \equiv j_{2} \bmod p^{\prime}$, we have
 \begin{align}
   (\ref{eq:A})
      & = \varepsilon \gcd(|p|, k) 
             \sum_{\substack{0 \leq i \leq |q| - 1, \\ 0 \leq j \leq p^{\prime} - 1}}
             \Phi_{o} \left( \left( w_{ij}, e^{\theta \sqrt{-1}} \right),
                              \left( w_{i, [i + 1]_{p^{\prime}}}, e^{\theta \sqrt{-1}} \right) 
                      \right)
              \notag \\
      & = \varepsilon \gcd (|p|, k)
           \left(
            \  \sum_{\substack{0 \leq i \leq |q| - 1, \\ 0 \leq j \leq p^{\prime} - 1}}
             - s(o, w_{ij}, w_{i,[i + 1]_{p^{\prime}}}) \right.
             \notag \\
      & \hspace{1.3em} \quad \quad \quad \quad \quad \quad \left. +\sum_{\substack{0 \leq i \leq |q| - 1, \\ 0 \leq j \leq p^{\prime} - 1}}
              s \left( o, ( w_{ij} - w_{i, [i + 1]_{p^{\prime}}}) e^{\theta \sqrt{-1}} + w_{i, [i + 1]_{p^{\prime}}}, w_{i, [i + 1]_{p^{\prime}}} \right)
           \right) \notag \\
      & = \varepsilon \gcd (|p|, k)
           \left(
            \  \sum_{\substack{0 \leq i \leq |q| - 1, \\ 0 \leq j \leq p^{\prime} - 1}}
             - s(o, w_{ij}, w_{i,[i + 1]_{p^{\prime}}}) \right.
             \notag \\
      & \hspace{1.3em} \quad \quad \quad \quad \quad \quad \left. +\sum_{\substack{0 \leq i \leq |q| - 1, \\ 0 \leq j \leq p^{\prime} - 1}}
              s(o, w_{[i + 1]_{|q|}, j}, w_{i, [i + 1]_{p^{\prime}}})
           \right).
                     \label{eq:B}
 \end{align}
 Since the regular polygon  
 $(w_{i, [i]_{p^{\prime}}}, w_{i, [i + 1]_{p^{\prime}}}, \cdots, w_{i, [i + |p| - 1]_{p^{\prime}}})$ 
 of type $(|p|, k)$ is obtained from $P_{0}$ by the $\frac{2 \pi li}{|q|}$-rotation about $o$, we have
 \begin{align}
   (\ref{eq:B}) & = \varepsilon |q| \gcd (|p|, k) 
           \left( \ 
            \sum_{0 \leq j \leq p^{\prime} - 1}  - s(o, w_{0j}, w_{01}) \hspace{1.5ex}
            +\sum_{0 \leq j \leq p^{\prime}-1} s(o, w_{1j}, w_{01}) 
           \right)
            \label{eq:C}
 \end{align}
 in light of Lemma \ref{sign_area}.
 We furthermore have $s(o, w_{1j}, w_{01}) = - s(o, w_{0, [2 - j]_{p^{\prime}}}, w_{01})$, because $(o, w_{0, [2 - j]_{p^{\prime}}}, w_{01})$ is obtained from $(o, w_{1j}, w_{01})$ by the reflection through the line $ow_{01}$.
 Thus, we have
 \begin{align}
  (\ref{eq:C}) & = \varepsilon |q| \gcd (|p|, k) 
           \left(
            \ \sum_{0 \leq j \leq p^{\prime} - 1}
              - s(o, w_{0j}, w_{01}) \hspace{1.5ex}
             +\sum_{0 \leq j \leq p^{\prime} - 1}
              - s(o, w_{0, [2 - j]_{p^{\prime}}}, w_{01})
          \right)
           \notag \\
      & = \varepsilon |q| \gcd (|p|,k) 
           \left(
            \ \sum_{0 \leq j \leq p^{\prime} - 1}
             - s(o, w_{0j}, w_{01}) \hspace{1.5ex}
            +\sum_{0 \leq j \leq p^{\prime} - 1}
             - s(o, w_{0j}, w_{01})
          \right) 
           \notag \\
      & = \varepsilon |q| \gcd (|p|, k) \cdot
            2 \sum_{0 \leq j \leq p^{\prime} - 1} s(o, w_{01}, w_{0j}).
            \label{eq:D}
 \end{align}
 We may rewrite 
 $\displaystyle 2 \sum_{0 \leq j \leq p^{\prime} - 1} s(o, w_{01}, w_{0j})$ as follows.
 \begin{align}
  & \hspace{1.3em} 2\sum_{0 \leq j \leq p^{\prime}-1} s(o, w_{01}, w_{0j})
     \notag \\
  & = 2s(o, w_{01}, w_{00})
     + 2s(o, w_{01}, w_{01})
     + 2 \sum_{2 \leq j \leq p^{\prime} - 1} s(o, w_{01}, w_{0j})
      \notag \\
  & = - 2s(o, w_{00}, w_{01})
     + \sum_{2 \leq j \leq p^{\prime} - 1} 
     \left( s(o, w_{01}, w_{0j}) + s(o, w_{01}, w_{0, [1 - j]_{p^{\prime}}}) 
     \right).
     \label{eq:E}
 \end{align}
 Since $(o, w_{01}, w_{0, [1 - j]_{p^{\prime}}})$ is obtained from $(o, w_{00}, w_{0j})$ by the reflection  through the perpendicular bisector of the line segment $w_{00} w_{01}$, we have $s(o, w_{01}, w_{0, [1 - j]_{p^{\prime}}}) = - s (o, w_{00}, w_{0j})$.
 Thus, we have
 \begin{align}
  (\ref{eq:E}) & = - 2s(o, w_{00}, w_{01})
                   + \sum_{2 \leq j \leq p^{\prime} - 1} 
                   \left( 
                    s(o, w_{01}, w_{0j}) - s(o, w_{00}, w_{0j}) 
                   \right)
                   \notag \\
               & = - 2s(o, w_{00}, w_{01})
                   + \sum_{2 \leq j \leq p^{\prime} - 1} S(o, w_{01}, w_{0j}, w_{00}).
                   \label{eq:F}
 \end{align} 
 Let $o^{\prime}$ be the center of (the convex hull of the vertices of) $P_{0}$ and 
 $o_{j}$ the point obtained from $o$ by the $\frac{2 \pi k^{\prime} j}{p^{\prime}}$-rotation 
 about $o^{\prime}$ for each $j$ ($0 \leq j \leq p^{\prime} - 1$).
 Then obviously $(o, w_{01}, w_{0j}, w_{00})$ is congruent to $(o_{[2 - j]_{p^{\prime}}}, w_{0, [3 - j]_{p^{\prime}}}, w_{02}, w_{0, [2 - j]_{p^{\prime}}})$ preserving the orders of vertices.
 Thus, remarking that $(o_{[2 - j]_{p^{\prime}}}, w_{0, [2 - j]_{p^{\prime}}}, w_{0, [3 - j]_{p^{\prime}}})$ is congruent to $(o, w_{00}, w_{01})$ preserving the orders of vertices, we have
 \begin{align}
  (\ref{eq:F}) & = - 2s(o, w_{00}, w_{01})
                   + \sum_{2 \leq j \leq p^{\prime} - 1} 
                    S(o_{[2 - j]_{p^{\prime}}}, w_{0, [3 - j]_{p^{\prime}}},
                             w_{02}, w_{0, [2 - j]_{p^{\prime}}})
                    \notag \\
               & = - 2s(o, w_{00}, w_{01})
                   + \sum_{2 \leq j \leq p^{\prime} - 1}
                    s(o_{[2 - j]_{p^{\prime}}}, w_{0, [3 - j]_{p^{\prime}}},
                             w_{0, [2 - j]_{p^{\prime}}})
                    \notag \\
               & \hspace{1.3em} \quad + \sum_{2 \leq j \leq p^{\prime} - 1}
                    s(w_{02}, w_{0, [2 - j]_{p^{\prime}}}, w_{0, [3 - j]_{p^{\prime}}})
                    \notag \\
               & = - 2s(o, w_{00}, w_{01})
                   - \sum_{2 \leq j \leq p^{\prime} - 1}
                    s(o_{[2 - j]_{p^{\prime}}}, w_{0, [2 - j]_{p^{\prime}}},
                             w_{0, [3 - j]_{p^{\prime}}})
                   + \frac{S(P_{0})}{\gcd(|p|, k)}
                    \notag \\
               & = - p^{\prime} \cdot s(o, w_{00}, w_{01})
                   + \frac{S(P_{0})}{\gcd(|p|, k)}
                    \notag \\
               & = - p^{\prime} \cdot \frac{S(Q)}{|q|}
                   + \frac{S(P_{0})}{\gcd(|p|, k)}.
                   \label{eq:G}
 \end{align}
 We obtain the claim substituting (\ref{eq:G}) in (\ref{eq:D}).
\end{proof}
 
\noindent
We note in the above proof that, since we have
\begin{align*}
& \hspace{2.5ex} \sum_{\substack{0 \leq i \leq |q| - 1, \\ 0 \leq j \leq p^{\prime} - 1}}
             \left(
             - s(o, w_{ij}, w_{i, [i + 1]_{p^{\prime}}})
             + s(o, w_{[i + 1]_{|q|}, j}, w_{i, [i + 1]_{p^{\prime}}})
             \right) \\
&=\sum_{\substack{0 \leq i \leq |q| - 1, \\ 0 \leq j \leq p^{\prime} - 1}}
             \left(
             - s(o, w_{ij}, w_{i, [i + 1]_{p^{\prime}}})
             + s(o, w_{[i + 1]_{|q|}, j}, w_{i, [i + 1]_{p^{\prime}}})
             \right. \\
&             \hspace{20ex}
             \left.
             - s(o, w_{[i + 1]_{|q|}, j}, w_{ij})
             + s(o, w_{[i + 1]_{|q|}, j}, w_{ij})
             \right) \\
&=\sum_{\substack{0 \leq i \leq |q| - 1, \\ 0 \leq j \leq p^{\prime} - 1}}
             \left(
             - s(o, w_{ij}, w_{i, [i + 1]_{p^{\prime}}})
             - s(o, w_{i, [i + 1]_{p^{\prime}}}, w_{[i + 1]_{|q|}, j})
             \right. \\
&             \hspace{20ex}
             \left.
             - s(o, w_{[i + 1]_{|q|}, j}, w_{ij})
             - s(o, w_{ij}, w_{[i + 1]_{|q|}, j})
             \right) \\
&= - \sum_{\substack{0 \leq i \leq |q| - 1, \\ 0 \leq j \leq p^{\prime} - 1}}
             S(w_{ij}, w_{i, [i + 1]_{p^{\prime}}}, w_{[i + 1]_{|q|}, j})
   - \sum_{\substack{0 \leq j \leq p^{\prime} - 1}}
             S(w_{0j}, w_{1j}, \cdots, w_{|q| - 1, j}),
\end{align*}
the value of (\ref{eq:B}), that is the value of $W_{\Phi_{o}} \left( D(p, q), \mathscr{C} \right)$, does not depend on the choice of $o$
\footnote{We can check that, for any 2-cycle $Z$ of $\RotE$, the value of $\Phi_{o}(Z)$ does not depend on the choice of $o$ in a similar way.}.
Thus, we write $W_{\Phi_{o}} \left( D(p, q), \mathscr{C} \right)$ as $W_{\Phi}(D(p, q), \mathscr{C})$ in the remaining.

\begin{proof}[Proof of Theorem \ref{necessary_condition_2}]
 Let $d$ and $d^{\prime}$ be the side lengths of $Q$ and $Q^{\prime}$ respectively.
 Then, in light of Proposition \ref{weight_of_Dpq}, we have 
 \[
  W_{\Phi}(D(p, q), \mathscr{C}^{\prime})
  = \left( \frac{d^{\prime}}{d} \right)^{2} W_{\Phi}(D(p, q), \mathscr{C}).
 \]
 We can check by a straightforward calculation that $W_{\Phi}(D(p, q), \mathscr{C}) \neq 0$.
 Thus, we know that $W_{\Phi}(D(p, q), \mathscr{C}^{\prime}) \neq W_{\Phi}(D(p, q), \mathscr{C})$ 
 if $d^{\prime} \neq d$.
\end{proof}

\section{Proof of Theorem \ref{main_thm}}
\label{sec:proof_of_main_theorem}

We devote this section to proving Theorem \ref{main_thm}.
Throughout this section, we let $p^{\prime} = \frac{|p|}{\gcd(|p|, k)}$, $q^{\prime} = \frac{|q|}{\gcd(|q|, l)}$, $k^{\prime} = \frac{k}{\gcd(|p|, k)}$, $l^{\prime} = \frac{l}{\gcd(|q|, l)}$ and $\theta = \theta(|p|, k; |q|, l)$.
Furthermore, we assume that $\mathscr{C}_{0}(a_{ij}) = \mathscr{C}(Q^{0}, P^{0}_{0})(a_{ij}) = (w_{i, [i + j + 1]_{|p|}}, e^{\theta \sqrt{-1}})$ and the side length of $Q^{0}$ is equal to $1$ scaling the coordinate system of the complex plane if necessary.
In light of Theorems \ref{necessary_condition_1} and \ref{necessary_condition_2}, we only need to investigate non-trivial $\RotE$-colorings of $D(p, q)$ derived from $(|p|, k; |q|, l)$-trochoids around polygons whose side lengths are $1$ to determine the R-equivalence class of $\mathscr{C}_{0}$.

Let $\alpha$ be $\frac{p^{\prime}q^{\prime}}{2}$ if $p^{\prime}q^{\prime}$ is even, otherwise $p^{\prime}q^{\prime}$, and consider the set $\mathscr{V}$ consisting of the $2 \alpha$ complex numbers
\[
v_{\sigma} = \cos \left( \arg (w_{01} - w_{00}) + \frac{\sigma \pi}{\alpha} \right) + \sqrt{-1} \sin \left( \arg (w_{01} - w_{00}) + \frac{\sigma \pi}{\alpha} \right),
\]
where $\sigma$ satisfies $0 \leq \sigma \leq 2 \alpha - 1$.
Furthermore, we inductively construct a set $\mathscr{W}_{1}$ as follows.
We first let $\mathscr{W}_{1}$ be the set consisting of $w_{00}$.
We next enlarge the set $\mathscr{W}_{1}$ by adding $w + v_{\sigma}$ to $\mathscr{W}_{1}$ for each $w \in \mathscr{W}_{1}$ and $v_{\sigma} \in \mathscr{V}$, repeatedly.
Consequently, $\mathscr{W}_{1}$ coincides with the set $\displaystyle \biggl\{ w_{00} + \sum_{0 \leq \sigma \leq 2 \alpha - 1} a_{\sigma} v_{\sigma} \biggm| a_{\sigma} \in \mathbb{Z} \biggr\}$.

\begin{proposition}
 \label{W_1}
 The set $\mathscr{W}_{1}$ coincides with the set, say $\mathscr{W}_{2}$, of points in $\mathbb{C}$ each of which some coloring obtained from $\mathscr{C}_{0}$ by shifts and even times of switches uses as the first component of a color of an arc of $D(p, q)$.
\end{proposition}

To prove this proposition, we prepare the following key deformation of $D(p,q)$.
Apply a switch, a shift, a switch and a shift to $D(p, q)$ in this order.
Then we obtain $D(p, q)$ again.
We call this sequence of deformations a \emph{fundamental deformation} of $D(p, q)$.
Let $\mathscr{C} = \mathscr{C}(Q, P_{0})$ be a non-trivial $\RotE$-coloring of $D(p, q)$ derived from a $(|p|, k; |q|, l)$-trochoid, and $\mathscr{C}^{\prime} = \mathscr{C}(Q^{\prime}, P^{\prime}_{0})$ the one of $D(p, q)$ obtained from $\mathscr{C}$ by a fundamental deformation, which is R-equivalent to $\mathscr{C}$.
Figure \ref{c1_c2} illustrates the diagrams of trochoids, from which colorings appearing in the sequence giving a fundamental deformation are derived, in a case of $p = 4$, $q = 3$, $k = 1$ and $l = 2$, for example.
We also call the deformation of the diagram of the trochoid from which $\mathscr{C}$ is derived to the one from which $\mathscr{C}^{\prime}$ is derived a \emph{fundamental deformation}.
It is easy to see that the diagram of the trochoid from which $\mathscr{C}^{\prime}$ is derived is obtained from the one from which $\mathscr{C}$ is derived by a ``$\theta$-rotation'' about a vertex $c$ of $Q$ (see Figure \ref{c1_c2}).
We call this vertex $c$ the \emph{center point} of the fundamental deformation of the diagram of the trochoid from which $\mathscr{C}$ is derived, or of $\mathscr{C}$ for short.
\begin{figure}[htbp]
 \centering
 \includegraphics[scale=0.5]{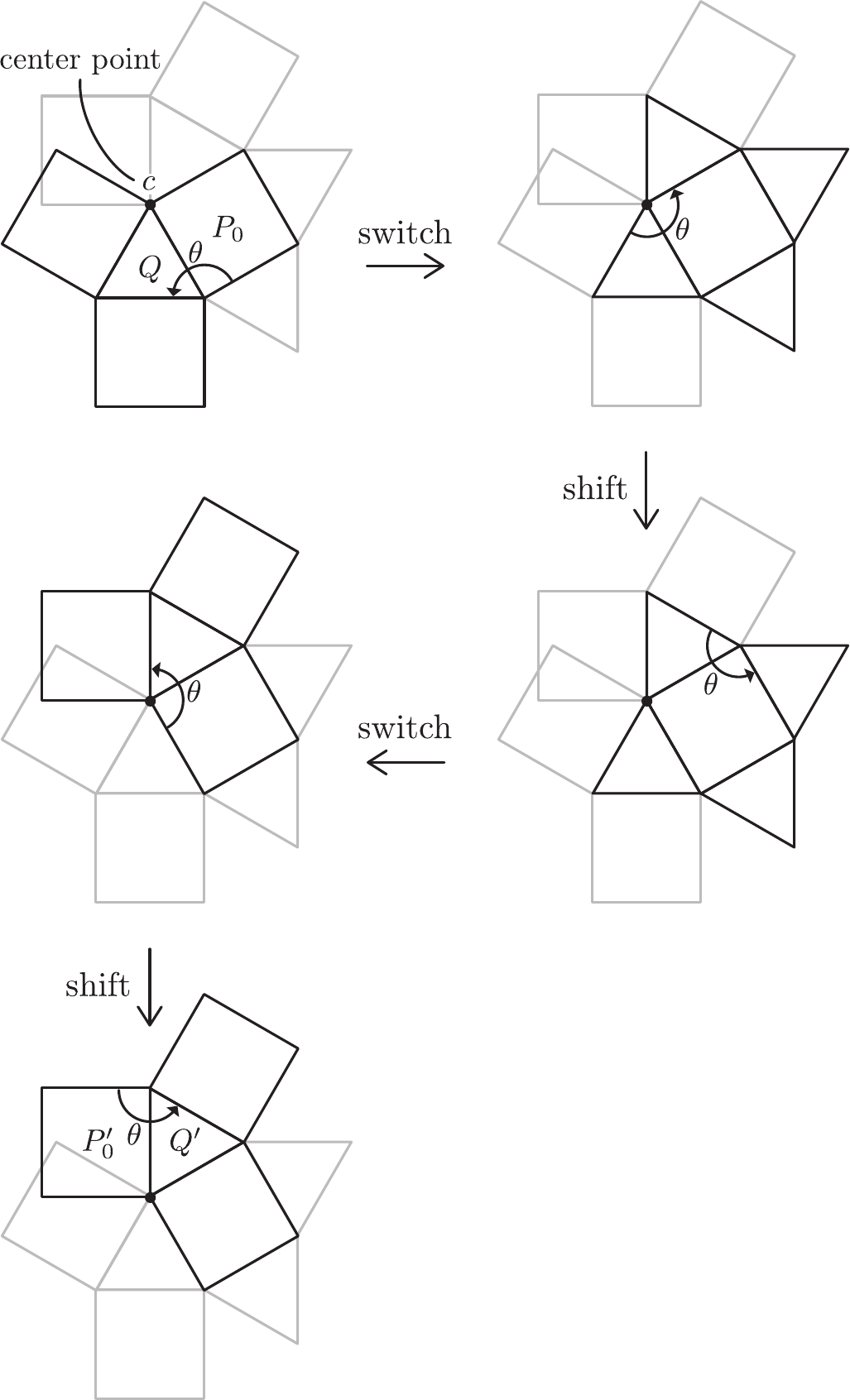}
 \caption{The diagrams of the trochoids, from which colorings appearing in the sequence giving a fundamental deformation are derived.}
 \label{c1_c2}
\end{figure}

\begin{remark}
 \label{fundamental_deformation_2}
 We can set any vertex of $Q$ as the center point of a fundamental deformation of a diagram of a trochoid, applying adequate shifts to the diagram of the trochoid from which $\mathscr{C}$ is derived before the fundamental deformation.
 It means that, for each vertex of $Q$, we have a non-trivial $\RotE$-coloring of $D(p,q)$ whose center point is the vertex obtained from $\mathscr{C}$ by several times of shifts.
\end{remark}

\begin{remark}
 \label{fundamental_deformation_1}
 We note that serial fundamental deformations of diagrams of trochoids share the same center point.
 Since $l^{\prime}p^{\prime} - k^{\prime}q^{\prime}$ and $p^{\prime}q^{\prime}$ are coprime, remarking that $\theta = \frac{2 (l^{\prime}p^{\prime} - k^{\prime}q^{\prime})}{p^{\prime}q^{\prime}} \pi$, we do not obtain the diagram of a $(|p|, k; |q|, l)$-trochoid again from itself by $n$-times of fundamental deformations if $0 < n < p^{\prime}q^{\prime}$.
\end{remark}

\begin{lemma}
 \label{fundamental_deformation_3}
 Let $\mathscr{C} = \mathscr{C}(Q, P_{0})$ be a non-trivial $\RotE$-coloring of $D(p, q)$ derived from a $(|p|, k; |q|, l)$-trochoid.
 Assume that $Q$ has an edge connecting the center point $c$ of $\mathscr{C}$ and $c + v_{\sigma}$ \textup{(}$v_{\sigma} \in \mathscr{V}$\textup{)}.
 Then, for each $v \in \mathscr{V}$, we obtain a non-trivial $\RotE$-coloring $\mathscr{C}(Q^{\prime}, P^{\prime}_{0})$ from $\mathscr{C}$ by several times of fundamental deformations so that $Q^{\prime}$ has an edge connecting $c$ and $c + v$.
\end{lemma}
 
\begin{proof}
 It is easy to see under the assumption that $Q$ also has an edge connecting $c$ and $c+v_{\tau}$.
 Here, $\tau$ is equal to either $\Bigl[ \sigma + (q^{\prime} - 1) \cdot \frac{2 l^{\prime} \alpha}{q^{\prime}} + \alpha \Bigr]_{2 \alpha}$ or $\Bigl[ \sigma - (q^{\prime} - 1) \cdot \frac{2 l^{\prime} \alpha}{q^{\prime}} - \alpha \Bigr]_{2 \alpha}$.
 We note that $v_{\tau}$ and $v_{\sigma}$ are the same if $q^{\prime} = 2$.
 By a similar argument in Remark \ref{fundamental_deformation_1}, applying $\theta$-rotations about $0$ to $v_{\sigma}$ and $v_{\tau}$ in a row, we obtain mutually different $p^{\prime}q^{\prime}$ complex numbers respectively.
 We let $\mathscr{V}_{\sigma}$ and $\mathscr{V}_{\tau}$ respectively denote the sets of those $p^{\prime}q^{\prime}$ complex numbers obtained from $v_{\sigma}$ and $v_{\tau}$.
 It is easy to see that a $\theta$-rotation about $0$ maps $v_{\mu} \in \mathscr{V}$ to $v_{[\mu + \beta]_{2 \alpha}}$, where $\beta = \frac{2 l^{\prime} \alpha}{q^{\prime}} - \frac{2 k^{\prime} \alpha}{p^{\prime}}$.
 Therefore, since $v_{\sigma}$ and $v_{\tau}$ are elements of $\mathscr{V}$, $\mathscr{V}_{\sigma}$ and $\mathscr{V}_{\tau}$ are subsets of $\mathscr{V}$.
 If $p^{\prime}q^{\prime}$ is even, comparing cardinality of the sets, we know that $\mathscr{V}_{\sigma} = \mathscr{V}_{\tau} = \mathscr{V}$.
 If $p^{\prime}q^{\prime}$ is odd, since $\sigma$ and $\tau$ have different parities (because $(q^{\prime} - 1) \cdot \frac{2 l^{\prime} \alpha}{q^{\prime}} + \alpha$ is odd) and $\beta$ is even, we know that $\mathscr{V}_{\sigma} \cap \mathscr{V}_{\tau} = \emptyset$.
 Thus, $\mathscr{V} = \mathscr{V}_{\sigma} \sqcup \mathscr{V}_{\tau}$.
 
 Suppose that $v$ is an element of $\mathscr{V}$.
 If $v$ is an element of $\mathscr{V}_{\sigma}$ (respectively of $\mathscr{V}_{\tau}$), applying adequate fundamental deformations to the diagram of the trochoid from which $\mathscr{C}$ is derived, we can map the edge connecting $c$ and $c + v_{\sigma}$ (respectively $c + v_{\tau}$) to the one connecting $c$ and $c+v$.
 It means that we can obtain a non-trivial $\RotE$-coloring $\mathscr{C}(Q^{\prime}, P^{\prime}_{0})$ of $D(p,q)$ from $\mathscr{C}$ by several times of fundamental deformations so that $Q^{\prime}$ has an edge connecting $c$ and $c + v$.
 Since $\mathscr{V} = \mathscr{V}_{\sigma} \cup \mathscr{V}_{\tau}$, we have the claim.
\end{proof}

In light of Lemma \ref{fundamental_deformation_3}, we immediately have the following corollary.

\begin{corollary}
 \label{fundamental_deformation_4}
 Assume that $\mathscr{C} = \mathscr{C}(Q, P_{0})$ is obtained from $\mathscr{C}_{0}$ by shifts and even times of switches, and $Q$ has an edge connecting the center point $c$ of $\mathscr{C}$ and $c + v_{\sigma}$ \textup{(}$v_{\sigma} \in \mathscr{V}$\textup{)}.
 Then, for each $v \in \mathscr{V}$, $c + v$ is an element of $\mathscr{W}_{2}$.
\end{corollary}

\begin{proof}[Proof of Propositon \ref{W_1}]
 We first prove inductively that $\mathscr{W}_{1}$ is a subset of $\mathscr{W}_{2}$.
 Since $w_{00}$ is the center point of $\mathscr{C}_{0}$ and $Q^{0}$ has an edge connecting $w_{00}$ and $w_{01} = w_{00} + v_{0}$, in light of Corollary \ref{fundamental_deformation_4}, we know that $w_{00} + v$ is an element of $\mathscr{W}_{2}$ for each $v \in \mathscr{V}$.
 
 For some $w \in \mathscr{W}_{1}$ and $v \in \mathscr{V}$, assume that we obtain a non-trivial $\RotE$-coloring $\mathscr{C} = \mathscr{C}(Q, P_{0})$ from $\mathscr{C}_{0}$ by shifts and even times of switches so that $w$ is the center point of $\mathscr{C}$ and $Q$ has an edge connecting $w$ and $w^{\prime} = w + v$.
 As mentioned in Remark \ref{fundamental_deformation_2}, applying adequate shifts to $\mathscr{C}$, we obtain a non-trivial $\RotE$-coloring $\mathscr{C}^{\prime} = \mathscr{C}(Q^{\prime}, P^{\prime}_{0})$ of $D(p,q)$ whose center point is $w^{\prime}$.
 We note that $Q^{\prime}$ has an edge connecting $w^{\prime}$ and $w$ since $Q^{\prime}$ coincides with $Q$.
 Since $w = w^{\prime} - v$ and $- v \in \mathscr{V}$, in light of Corollary \ref{fundamental_deformation_4}, we know that $w^{\prime} + v^{\prime}$ is an element of $\mathscr{W}_{2}$ for each $v^{\prime} \in \mathscr{V}$.
 
 We next prove inductively that $\mathscr{W}_{2}$ is a subset of $\mathscr{W}_{1}$.
 To do it, we consider the following condition for a non-trivial $\RotE$-coloring $\mathscr{C} = \mathscr{C}(Q, P_{0})$ of $D(p, q)$ or $D(q, p)$:
 \begin{itemize}
  \item[$(\heartsuit)$] Each first component of colors of the arcs of $D(p, q)$ or $D(q, p)$ by $\mathscr{C}$ is an element of $\mathscr{W}_{1}$.
  Furthermore, each edge of $Q$ and $P_{i}$ is parallel to some line segment $0v$ ($v \in \mathscr{V}$).
 \end{itemize}
 Since the line segment $w_{00}w_{01}$ is parallel to the line segment $0 v_{0}$, it is routine to see that the line segment $w_{i, [i + j]_{|p|}} w_{i, [i + j + 1]_{|p|}}$ is parallel to $0 v_{[\gamma]_{2 \alpha}}$, where $\gamma = i \cdot \frac{2 l^{\prime} \alpha}{q^{\prime}} + j \cdot \frac{2 k^{\prime} \alpha}{p^{\prime}}$.
 Therefore, since $w_{00}$ is an element of $\mathscr{W}_{1}$, we know that $w_{i, [i + j + 1]_{|p|}}$ is an element of $\mathscr{W}_{1}$.
 Thus, $\mathscr{C}_{0}$ satisfies the condition $(\heartsuit)$.
 
 Let $\mathscr{C} = \mathscr{C}(Q, P_{0})$ be a non-trivial $\RotE$-coloring of $D(p, q)$ or $D(q, p)$ obtained from $\mathscr{C}_{0}$ by shifts and $n$-times of switches.
 Assume that $\mathscr{C}$ satisfies the condition $(\heartsuit)$.
 Suppose that $\mathscr{C}^{\prime} = \mathscr{C}(Q^{\prime}, P^{\prime}_{0})$ is a non-trivial $\RotE$-coloring obtained from $\mathscr{C}$ by a shift or a switch.
 If $\mathscr{C}^{\prime}$ is obtained from $\mathscr{C}$ by a shift, since $Q^{\prime}$ and $P^{\prime}_{i}$ respectively coincide with $Q$ and $P_{i+1}$, $\mathscr{C}^{\prime}$ obviously satisfies the condition $(\heartsuit)$.
 Assume that $\mathscr{C}^{\prime}$ is obtained from $\mathscr{C}$ by a switch.
 We note that this is the $(n+1)$-th switch in total.
 Since $Q^{\prime}$ and $P_{0}$ have the same shape, the vertices of $Q^{\prime}$ are elements of $\mathscr{W}_{1}$ and each edge of $Q^{\prime}$ is parallel to some line segment $0 v$ ($v \in \mathscr{V}$).
 Since each $P^{\prime}_{i}$ shares an edge with $Q^{\prime}$, we may assume that the edge is parallel to some line segment $0 v_{\sigma}$ ($v_{\sigma} \in \mathscr{V}$).
 Then each edge of $P^{\prime}_{i}$ is parallel to one of the line segments $0 v_{[\sigma + \tau \delta]_{2 \alpha}}$, where $\delta = \frac{2 l^{\prime} \alpha}{q^{\prime}}$ and $0 \leq \tau \leq |q| - 1$ if $n$ is even, otherwise $\delta = \frac{2 k^{\prime} \alpha}{p^{\prime}}$ and $0 \leq \tau \leq |p| - 1$.
 Therefore, $\mathscr{C}^{\prime}$ also satisfies the condition $(\heartsuit)$.
\end{proof}

\begin{proposition}
 \label{W_2}
 The set $\mathscr{W}_{2}$ coincides with the set, say $\mathscr{W}_{3}$, of points in $\mathbb{C}$ each of which some coloring obtained from $\mathscr{C}_{0}$ by a finite sequence of Reidemeister moves and planar isotopies uses as the first component of a color of an arc of a diagram of the $(p, q)$-torus knot.
\end{proposition}

\begin{proof}
 It is clear that $\mathscr{W}_{2}$ is a subset of $\mathscr{W}_{3}$.
 Therefore, we prove inductively that $\mathscr{W}_{3}$ is a subset of $\mathscr{W}_{2}$.
 To do it, we consider the following condition for a non-trivial $\RotE$-coloring $\mathscr{C}$ of a diagram $D$ of the $(p, q)$-torus knot:
 \begin{itemize}
  \item [$(\diamondsuit)$] Each first component of colors of the arcs of $D$ by $\mathscr{C}$ is an element of $\mathscr{W}_{2}$.
 \end{itemize}
 We first note that, by the definition of $\mathscr{W}_{2}$, $\mathscr{C}_{0}$ satisfies the condition $(\diamondsuit)$.
 Let $\mathscr{C}$ be a non-trivial $\RotE$-coloring of a diagram $D$ obtained from $\mathscr{C}_{0}$ by a finite sequence of Reidemeister moves and planar isotopies, and assume that $\mathscr{C}$ satisfies the condition $(\diamondsuit)$.
 Suppose that $\mathscr{C}^{\prime}$ is a non-trivial $\RotE$-coloring of a diagram $D^{\prime}$ obtained from $\mathscr{C}$ by a Reidemeister move or a planar isotopy. 
 If $\mathscr{C}^{\prime}$ is obtained from $\mathscr{C}$ by a planar isotopy, an RI or an RII which decreases the number of crossings, since each color of the arcs of $D^{\prime}$ by $\mathscr{C}^{\prime}$ is already used as a color of some arc of $D$ by $\mathscr{C}$, $\mathscr{C}^{\prime}$ also satisfies the condition $(\diamondsuit)$.
 Otherwise, although we may have a color of some arc of $D^{\prime}$ by $\mathscr{C}^{\prime}$ which is different from any colors of the arcs of $D$ by $\mathscr{C}$, that color can be written as $x \, {\ast}^{\pm 1} \, y$ with some colors $x$, $y$ of arcs of $D$ by $\mathscr{C}$.
 Let $\widehat{x}$, $\widehat{y}$ and $\widehat{x \, {\ast}^{\pm 1} \, y}$ denote the first components of $x$, $y$  and $x \, {\ast}^{\pm 1} \, y$ respectively.
 Since $\mathscr{W}_{2} = \mathscr{W}_{1}$ by Proposition \ref{W_1}, $\widehat{x}$ can be written with some integers $a_{\sigma}$ as 
 \[
 \widehat{x} = \widehat{y} + \sum_{0 \leq \sigma \leq 2 \alpha - 1} a_{\sigma} v_{\sigma}.
 \]
 Since $\widehat{x \, {\ast}^{\pm 1} \, y}$ is obtained from $\widehat{x}$ by the $\pm \theta$-rotation about $\widehat{y}$, we have
 \[
 \widehat{x \, {\ast}^{\pm 1} \, y} = \widehat{y} + \sum_{0 \leq \sigma \leq 2 \alpha -1} a_{\sigma} v_{[\sigma \pm \beta]_{2 \alpha}}.
 \]
 Here, $\beta$ is the one defined in the proof of Lemma \ref{fundamental_deformation_3}.
 It means that $\widehat{x \, {\ast}^{\pm 1} \, y}$ is an element of $\mathscr{W}_{1} = \mathscr{W}_{2}$.
 Therefore, $\mathscr{C}^{\prime}$ also satisfies the condition $(\diamondsuit)$.
\end{proof}

By Propositions \ref{W_1} and \ref{W_2}, we obviously have the following claim.

\begin{corollary}
 \label{W_2_and_W}
 For each $w, w^{\prime} \in \mathscr{W}_{3}$, $w^{\prime}$ can be written with some integers $a_{\sigma}$ as
 \[
 w^{\prime} = w + \sum_{0 \leq \sigma \leq 2 \alpha - 1} a_{\sigma} v_{\sigma}.
 \]
\end{corollary}

Furthermore, by Corollary \ref{W_2_and_W} and Proposition \ref{prop:main} in Appendix A, we straightforwardly have the following corollary.

\begin{corollary}
 \label{point_number}
 For each $w \in \mathscr{W}_{3}$, there are exactly $2 \alpha$ points of $\mathscr{W}_{3}$ on the unit circle centered at $w$, which are $w + v_{\sigma}$ \textup{(}$0 \leq \sigma \leq 2 \alpha - 1$\textup{)}.
\end{corollary}

\begin{proof}[Proof of Theorem \ref{main_thm}]
 Suppose that $\mathscr{C} = \mathscr{C}(Q, P_{0})$ is a non-trivial $\RotE$-coloring of $D(p, q)$ being R-equivalent to $\mathscr{C}_{0}$.
 Then, by definition, the vertices of $Q$ and each $P_{i}$ are elements of $W_{3}$.
 Since the side length of $Q$ is $1$ (as mentioned in the first paragraph of this section), in light of Corollary \ref{point_number}, each edge of $Q$ should connect some $w \in \mathscr{W}_{3}$ and $w + v$ ($v \in \mathscr{V}$).
 
 Conversely, for some $w \in \mathscr{W}_{3}$, let $w_{\sigma} = w + {v}_{\sigma}$ (${v}_{\sigma} \in \mathscr{V}$).
 Then, for each $\sigma$, there are at most two regular polygons of type $(|q|,l)$ having the line segment $ww_{\sigma}$ as a side. 
 One is $Q^{+}$ having $w w_{[\sigma + \varepsilon]_{2\alpha}}$ as another side, and the other $Q^{-}$ having  $w w_{[\sigma - \varepsilon]_{2\alpha}}$ as another side, where $\varepsilon$ is $(q^{\prime} - 1) \cdot \frac{2 l^{\prime} \alpha}{q^{\prime}} + \alpha$.
 We note that $Q^{+}$ coincides with $Q^{-}$ if and only if $q^{\prime} = 2$.
 Let $\mathscr{C}^{+}$ and $\mathscr{C}^{-}$ be non-trivial $\RotE$-colorings of $D(p,q)$ derived from $(|p|,k;|q|,l)$-trochoids around $Q^{+}$ and $Q^{-}$ respectively.
 Then the argument that $\mathscr{W}_{1}$ is a subset of $\mathscr{W}_{2}$ in the proof of Proposition \ref{W_1} lets us know that at least $\mathscr{C}^{+}$ or $\mathscr{C}^{-}$ is obtained from $\mathscr{C}_{0}$ by shifts and even times of switches.
 Assume that we obtain $\mathscr{C}^{+}$ from $\mathscr{C}_{0}$ by shifts and even times of switches.
 By Remark \ref{fundamental_deformation_2}, applying adequate shifts to $\mathscr{C}^{+}$, we obtain a non-trivial $\RotE$-coloring $\mathscr{C}^{+}_{\ast}$ of $D(p,q)$ whose center point is $w$.
 Choose ${v}_{[\sigma + \varepsilon]_{2 \alpha}}$ as $v_{\tau}$ in the proof of Lemma \ref{fundamental_deformation_3}.
 Since $p^{\prime}q^{\prime}$ is even, we know that $\mathscr{V}_{\sigma} = \mathscr{V}_{\tau} = \mathscr{V}$.
 Thus, we can obtain a non-trivial $\RotE$-coloring $\mathscr{C}^{-}_{\ast}$ of $D(p,q)$, which is derived from a trochoid around $Q^{-}$, from $\mathscr{C}^{+}_{\ast}$ by serial fundamental deformations.
 Applying adequate shifts to $\mathscr{C}^{-}_{\ast}$, we can obtain $\mathscr{C}^{-}$.
 Therefore, we can obtain $\mathscr{C}^{-}$ from $\mathscr{C}^{+}$ by shifts and even times of switches.
 In a similar way, we can obtain $\mathscr{C}^{+}$ from $\mathscr{C}^{-}$ by shifts and even times of switches, even if we obtain $\mathscr{C}^{-}$ from $\mathscr{C}_{0}$.
 Therefore, both $\mathscr{C}^{+}$ and $\mathscr{C}^{-}$ are obtained from $\mathscr{C}_{0}$ by shifts and even times of switches.
 It means that both $\mathscr{C}^{+}$ and $\mathscr{C}^{-}$ are R-equivalent to $\mathscr{C}_{0}$.
 
 In conclusion, all colorings of $D(p, q)$ being R-equivalent to $\mathscr{C}_{0}$ are obtained from $\mathscr{C}_{0}$ by shifts and even times of switches.
\end{proof}

\begin{remark}
 \label{odd_case}
 If $p^{\prime}q^{\prime}$ is odd, we never obtain $\mathscr{C}^{-}_{\ast}$ from $\mathscr{C}^{+}_{\ast}$ by serial fundamental deformations, because $\mathscr{V}_{\sigma} \cap \mathscr{V}_{\tau} = \emptyset$.
 However, we might obtain $\mathscr{C}^{-}_{\ast}$ from $\mathscr{C}^{+}_{\ast}$ by another finite sequence of Reidemeister moves and planar isotopies.
 Therefore, we could not determine the R-equivalence class of $\mathscr{C}_{0}$ completely so far in this case.
\end{remark}

\section*{Acknowledgments}
\label{sec:acknowledgments}

The author would like to express her sincere gratitude to her supervisor Professor Ayumu Inoue for his hearty encouragements and helpful suggestions.
She also would like to express her thanks to Professor Takashi Hara who kindly lets her know valuable information and provides the appendix to this paper.

\appendix
\section{On cyclotomic integers of absolute value $1$ \\ by Takashi Hara\footnote{t-hara@tsuda.ac.jp} (Department of Mathematics, Tsuda University)}
\label{sec.appendix}

For a natural number $n$, let $\zeta_n$ denote a particular $n$-th root of unity $\exp(2\pi \sqrt{-1}/n)\in \mathbb{C}$. In this appendix, we verify the following claim yielding Corollary  \ref{point_number}. 

\begin{proposition} \label{prop:main}
Let $\alpha$ be a complex number of the form
\begin{align} \label{eq:cyclo_int}
 \alpha=\sum_{k=0}^{n-1} c_k\zeta_n^k =c_0+c_1\zeta_n+c_2\zeta_n^2+\cdots+c_{n-1}\zeta_n^{n-1} \qquad \text{for } c_0,c_1,\dotsc, c_{n-1}\in \mathbb{Z}.
\end{align} 
Then $\alpha$ is a root of unity if the absolute value $\lvert \alpha \rvert$ of $\alpha$ equals $1$.
More precisely $\alpha$ is an $n$-th root of unity if $n$ is even and a $2n$-th root of unity if $n$ is odd (see Lemma~\ref{lem:unit_torsion}).
\end{proposition} 

A complex number $\alpha$ of the form \eqref{eq:cyclo_int} is called a {\em cyclotomic integer}, which is a typical object dealt with in algebraic number theory. Indeed, one may deduce Proposition \ref{prop:main} by standard algebraic arguments. In the following, we shall give the proof of Proposition~\ref{prop:main} in some detail, for the sake of those who are not familiar with arithmetic theory. We refer the readers to \cite[Part II]{Cox} for basic notion in field theory.

\subsection{Materials from algebraic number theory}

In the following, we regard every algebraic extension of $\mathbb{Q}$ as a subfield of the complex number field $\mathbb{C}$. Finite extension fields of $\mathbb{Q}$ are called ({\em algebraic}) {\em number fields}. An algebraic number $\alpha \in \mathbb{C}$ is said to be {\em integral} (over $\mathbb{Q}$) if it is a root of a certain monic polynomial $f(x)$ with coefficients in $\mathbb{Z}$. For a number field $K$, we define $\mathcal{O}_K$ as a set of all the integral numbers contained in $K$, which is called the {\em ring of integers} of $K$. The group of multiplicatively invertible elements $\mathcal{O}_K^{\times}$ of $K$ is called the {\em unit group} of $K$.  

The structure of the unit group of a number field is determined by Dirichlet, which is one of the fundamental results of algebraic number theory. To state Dirichlet's result precisely, let us prepare several notation. A field embedding $K\hookrightarrow \mathbb{C}$ of a number field $K$ into $\mathbb{C}$ is called a {\em real embedding} if its image is contained in the real number field $\mathbb{R}$; otherwise it is called a {\em complex embedding}. For a complex embedding $\iota \colon K\hookrightarrow \mathbb{C}$, one could define its {\em conjugate} $\overline{\iota}\colon K\hookrightarrow \mathbb{C}$ by setting $\overline{\iota}(\alpha):=\overline{\iota(\alpha)}$ for every $\alpha\in K$; here $\overline{\iota(\alpha)}$ denotes the complex conjugate of $\iota(\alpha)\in \mathbb{C}$. We call $\iota$ and $\overline{\iota}$ a {\em conjugate pair} of complex embeddings. The number of real embeddings of $K$ and that of conjugate pairs of complex embeddings are respectively denoted as $r_{K,1}$ and $r_{K,2}$. Then, since the total number of embeddings of $K$ into $\mathbb{C}$ is known to equal the extension degree $[K:\mathbb{Q}]$ of $K$ over $\mathbb{Q}$ (see \cite[Corollary of Theorem 50]{Marcus} for details), we obtain a basic identity
\begin{align} \label{eq:emb_degree}
 r_{K,1}+2r_{K,2}=[K:\mathbb{Q}].
\end{align}

\begin{theorem}[Dirichlet's unit theorem] \label{thm:Dirichlet}
Let $K$ be a number field and $\mathcal{O}_K^\times$ the unit group of $K$. Furthermore, let $r_{K,1}$ and $r_{K,2}$ be as above. Then $\mathcal{O}_K^\times$ is a finitely generated abelian group of rank $\rho_K:=r_{K,1}+r_{K,2}-1$. More precisely, let $W_K$ denote the multiplicative group consisting of all the roots of unity in $K$. Then there exist $\rho_K$ units $\varepsilon_1,\varepsilon_2,\dotsc,\varepsilon_{\rho_K}\in \mathcal{O}_K^\times$ such that any unit $u\in \mathcal{O}_K^\times$ is described as 
\begin{align*}
 u=\zeta \varepsilon_1^{e_1}\varepsilon_2^{e_2}\cdots \varepsilon_{\rho_K}^{e_{\rho_K}} \qquad \text{for } \zeta\in W_K,\; e_1,e_2,\dotsc,e_{\rho_K}\in \mathbb{Z}
\end{align*}
in a unique manner.
\end{theorem}

The $\rho_K$ units $\varepsilon_1,\varepsilon_2,\dotsc,\varepsilon_{\rho_K}$ are often called the {\em fundamental units} of $K$ (though they are not determined uniquely).

\begin{proof}
 Dirichlet's original proof can be seen in \cite[\S 183]{Dirichlet}. For a modern proof using Minkowski's {\em Geometry of Numbers}, see \cite[Theorem 38]{Marcus} for example.
\end{proof}

\subsection{Cyclotomic fields and its maximal real subfields}

For a natural number $n$, we call a field $\mathbb{Q}(\zeta_n)$ obtained by adjoining the primitive $n$-th root of unity $\zeta_n$ to $\mathbb{Q}$ the {\em $n$-th cyclotomic field}, which is known to be a Galois extension of $\mathbb{Q}$ of degree $\varphi(n)$; here $\varphi(n):=\#\bigl(\mathbb{Z}/n\mathbb{Z}\bigr)^\times$ denotes Euler's totient function (see \cite[Corollary 9.1.10]{Cox}). It is well known that the ring of integers $\mathcal{O}_{\mathbb{Q}(\zeta_n)}$ of $\mathbb{Q}(\zeta_n)$ coincides with 
\begin{align*}
 \mathbb{Z}[\zeta_n]:=\{ c_0+c_1\zeta_n+c_2\zeta_n^2+\cdots +c_{n-1}\zeta_n^{n-1} \mid c_0,c_1,\dotsc,c_{n-1}\in \mathbb{Z} \}
\end{align*}
(see \cite[Corollary 2 of Theorem 12]{Marcus}). The Galois group $\mathrm{Gal}(\mathbb{Q}(\zeta_n)/\mathbb{Q})$ is explicitly described via the isomorphism
\begin{align*}
 \bigl(\mathbb{Z}/n\mathbb{Z}\bigr)^\times \xrightarrow{\;\; \cong \;\;} \mathrm{Gal}(\mathbb{Q}(\zeta_n)/\mathbb{Q})\, ; \, a \!\!\!\pmod{n} \mapsto \sigma_a,
\end{align*}
where $\sigma_a$ is a unique element of $\mathrm{Gal}(\mathbb{Q}(\zeta_n)/\mathbb{Q})$ characterised by $\sigma_a(\zeta_n)=\zeta_n^a$ (see \cite[Theorem~9.1.11]{Cox}). Then one readily observes that $\sigma_{-1}\in \mathrm{Gal}(\mathbb{Q}(\zeta_n)/\mathbb{Q})$ is induced by the complex conjugation on $\mathbb{C}$ (recall that we regard $\mathbb{Q}(\zeta_n)$ as a subfield of $\mathbb{C}$), and it generates a subgroup $\langle \sigma_{-1}\rangle=\{\mathrm{id}, \sigma_{-1}\}$ of cardinality $2$ in $\mathrm{Gal}(\mathbb{Q}(\zeta_n)/\mathbb{Q})$. Let $\mathbb{Q}(\zeta_n)^+=\mathbb{Q}(\zeta_n+\zeta_n^{-1})$ denote the subfield of $\mathbb{Q}(\zeta_n)$ corresponding to $\langle \sigma_{-1}\rangle$; it is a Galois extension of $\mathbb{Q}$ of degree $\varphi(n)/2$ by definition. For later use, we compute $r_{K,1}$ and $r_{K,2}$ for $K=\mathbb{Q}(\zeta_n)$ and $\mathbb{Q}(\zeta_n)^+$. 

\begin{lemma} \label{lem:r1r2}
 Let the notation be as above. Then we have $(r_{\mathbb{Q}(\zeta_n),1}, r_{\mathbb{Q}(\zeta_n),2})=(0,\varphi(n)/2)$ and $(r_{\mathbb{Q}(\zeta_n)^+,1}, r_{\mathbb{Q}(\zeta_n)^+,2})=(\varphi(n)/2,0)$.
\end{lemma}

\begin{proof}
 For a finite Galois extension $K$ over $\mathbb{Q}$, the set of embeddings of $K$ into $\mathbb{C}$ is described as $\{\iota_0\circ \sigma \mid \sigma\in \mathrm{Gal}(K/\mathbb{Q})\}$ where $\iota_0\colon K\hookrightarrow \mathbb{C}$ be the natural inclusion (see \cite[Theorem 52]{Marcus}). Suppose $K=\mathbb{Q}(\zeta_n)$. Then, for each $\sigma_a\in \mathrm{Gal}(\mathbb{Q}(\zeta_n)/\mathbb{Q})$, the embedding $\iota_0\circ \sigma_a$ sends $\zeta_n$ to an imaginary number $\zeta_n^a$; in other words, every embedding of $\mathbb{Q}(\zeta_n)$ into $\mathbb{C}$ is complex. This observation and the basic identity \eqref{eq:emb_degree} imply the first half of the statement. Next suppose $K=\mathbb{Q}(\zeta_n)^+=\mathbb{Q}(\zeta_n+\zeta_n^{-1})$. Then, for each $\sigma_a \! \pmod{\langle \sigma_{-1}\rangle} \in \mathrm{Gal}(\mathbb{Q}(\zeta_n)^+/\mathbb{Q})$, the embedding $\iota_0\circ \sigma_a$ sends $\zeta_n+\zeta_n^{-1}$ to a real number $\zeta_n^a+\zeta_n^{-a}$; in other words, every embedding of $\mathbb{Q}(\zeta_n)^+$ into $\mathbb{C}$ is real. This observation and the basic identity \eqref{eq:emb_degree} imply the latter half of the statement.
\end{proof}

\begin{lemma} \label{lem:unit_torsion}
 For a number field $K$, let $W_K$ denote the multiplicative group of all the roots of unity in $K$. Then we have $W_{\mathbb{Q}(\zeta_n)}=\{\pm \zeta_n^k \mid k=0,1,\ldots,n-1\}$ and $W_{\mathbb{Q}(\zeta_n)^+}=\{\pm 1\}$.
\end{lemma}

One could immediately deduce from Lemma~\ref{lem:unit_torsion} that 
\begin{align*}
 W_{\mathbb{Q}(\zeta_n)}=
\begin{cases}
 \mu_n(\mathbb{C}) & \text{if $n$ is even}, \\
 \mu_{2n}(\mathbb{C}) & \text{if $n$ is odd} 
\end{cases}
\end{align*}
holds where $\mu_m(\mathbb{C})$ denotes the multiplicative group of all the $m$-th roots of unity.

\begin{proof}[Proof of Lemma~$\ref{lem:unit_torsion}$]
 See \cite[Corollary 3 of Theorem 3]{Marcus} for the first half of the statement. The latter half is obvious since $\mathbb{Q}(\zeta_n)^+=\mathbb{Q}(\zeta_n+\zeta_n^{-1})$ is a subfield of $\mathbb{R}$.
\end{proof}

\subsection{Restatement via the norm map} 

The {\em norm map} $\mathrm{Nr}_{\mathbb{Q}(\zeta_n)/\mathbb{Q}(\zeta_n)^+}$ from $\mathbb{Q}(\zeta_n)$ to $\mathbb{Q}(\zeta_n)^+$ (see \cite[p.\ 15]{Marcus}) is described explicitly as 
\begin{align*} 
 \mathrm{Nr}_{\mathbb{Q}(\zeta_n)/\mathbb{Q}(\zeta_n)^+}(\alpha)=\alpha\cdot \sigma_{-1}(\alpha)=\alpha\overline{\alpha}=\lvert \alpha \rvert^2 \qquad \text{for }\alpha\in \mathbb{Q}(\zeta_n)^\times.
\end{align*}
Here $\lvert \, \cdot \, \rvert$ denotes the usual absolute value defined on $\mathbb{C}$. The norm map $\mathrm{Nr}_{\mathbb{Q}(\zeta_n)/\mathbb{Q}(\zeta_n)^+}$ induces a homomorphism of finitely generated abelian groups
\begin{align} \label{eq:norm}
 \mathrm{Nr}_{\mathbb{Q}(\zeta_n)/\mathbb{Q}(\zeta_n)^+} \colon \mathcal{O}_{\mathbb{Q}(\zeta_n)}^\times \longrightarrow \mathcal{O}_{\mathbb{Q}(\zeta_n)^+}^\times \, ; \, \alpha \mapsto \lvert \alpha \rvert^2.
\end{align}
Furthermore if $\alpha \in \mathbb{Z}[\zeta_n]=\mathcal{O}_{\mathbb{Q}(\zeta_n)}$ satisfies the equality $\mathrm{Nr}_{\mathbb{Q}(\zeta_n)/\mathbb{Q}(\zeta_n)^+}(\alpha)=1$,  it is indeed a unit of $\mathcal{O}_{\mathbb{Q}(\zeta_n)}$ since $\overline{\alpha} \in \mathcal{O}_{\mathbb{Q}(\zeta_n)}$ is the inverse of $\alpha$ due to \eqref{eq:norm}. From these observations and Lemma~\ref{lem:unit_torsion}, one readily sees that Proposition~\ref{prop:main} is deduced from the following claim.

\begin{proposition} \label{prop:main2}
 The kernel of the homomorphism \eqref{eq:norm} is $W_{\mathbb{Q}(\zeta_n)}$.
\end{proposition} 

In the rest of the appendix, we give the proof of Proposition~\ref{prop:main2}.

\subsection{Proof of Proposition~\ref{prop:main2}} For $K=\mathbb{Q}(\zeta_n)$ or $\mathbb{Q}(\zeta_n)^+$, let $\mathcal{E}_K$ denote the quotient group $\mathcal{O}_K^\times/W_K$; then $\mathcal{E}_K$ is a free abelian group since $W_K$ is indeed the torsion part of the finitely generated abelian group $\mathcal{O}_K^\times$. Furthermore both $\mathcal{E}_{\mathbb{Q}(\zeta_n)}$ and $\mathcal{E}_{\mathbb{Q}(\zeta_n)^+}$ are of rank $\rho_n:=\varphi(n)/2-1$ due to Theorem~\ref{thm:Dirichlet} and Lemma~\ref{lem:r1r2}. The norm map $\mathrm{Nr}_{\mathbb{Q}(\zeta_n)/\mathbb{Q}(\zeta_n)^+}$ then induces a homomorphism of free abelian groups $\overline{\mathrm{Nr}} \colon \mathcal{E}_{\mathbb{Q}(\zeta_n)}\rightarrow \mathcal{E}_{\mathbb{Q}(\zeta_n)^+}$. In order to verify Proposition~\ref{prop:main2}, it suffices to show that $\overline{\mathrm{Nr}}$ is injective. 

To study $\overline{\mathrm{Nr}}$, let us fix fundamental units $\varepsilon_1,\varepsilon_2,\dotsc,\varepsilon_{\rho_n}$ of $\mathbb{Q}(\zeta_n)$ and $\epsilon_1,\epsilon_2,\dotsc,\epsilon_{\rho_n}$ of $\mathbb{Q}(\zeta_n)^+$, respectively. We use the same symbols for their images in the torsion free quotients $\mathcal{E}_{\mathbb{Q}(\zeta_n)}$ (or $\mathcal{E}_{\mathbb{Q}(\zeta_n)^+}$) by abuse of notation. Let $A_{\overline{\mathrm{Nr}}}\in \mathrm{M}_{\rho_n}(\mathbb{Z})$ be the matrix presentation of $\overline{\mathrm{Nr}}$ with respect to these fundamental units, which is precisely defined as follows: for any $1\leq j\leq \rho_n$, let us describe $\overline{\mathrm{Nr}}(\varepsilon_j)$ by the fundamental units of $\mathbb{Q}(\zeta_n)^+$ as
\begin{align*}
 \overline{\mathrm{Nr}}(\varepsilon_j)=\epsilon_1^{a_{1j}}\epsilon_2^{a_{2j}}\cdots \epsilon_{\rho_n}^{a_{\rho_n\rho_n}} \qquad \text{for }a_1,a_2,\dotsc, a_{\rho_n}\in \mathbb{Z}.
\end{align*}
The matrix presentation $A_{\overline{\mathrm{Nr}}}$ of $\overline{\mathrm{Nr}}$ is then defined as $A_{\overline{\mathrm{Nr}}}=\bigl( a_{ij}\bigr)_{1\leq i,j\leq \rho_n}$. It is known that $\overline{\mathrm{Nr}}$ is injective if and only if the determinant $\det (A_{\overline{\mathrm{Nr}}})$ of $A_{\overline{\mathrm{Nr}}}$ is nonzero (see \cite[Proposition (3.17) (2)]{AW}). Now let $i \colon \mathbb{Q}(\zeta_n)^+\hookrightarrow \mathbb{Q}(\zeta_n)$ denote a natural inclusion, which induces a homomorphism $\overline{i}\colon \mathcal{E}_{\mathbb{Q}(\zeta_n)^+}\rightarrow \mathcal{E}_{\mathbb{Q}(\zeta_n)}$ of free abelian groups. Then the composition $\overline{\mathrm{Nr}}\circ \overline{i}$ is just the squaring map $x\mapsto x^2$ since one could calculate as
\begin{align*}
 \mathrm{Nr}_{\mathbb{Q}(\zeta_n)/\mathbb{Q}(\zeta_n)^+}\circ i(x)=x\cdot \sigma_{-1}(x)=x\cdot x=x^2 
\end{align*}
for $x\in \mathcal{O}_{\mathbb{Q}(\zeta_n)^+}^\times$. This fact is rewritten as $A_{\overline{\mathrm{Nr}}}A_{\overline{i}}=2I_{\rho_n}$, where $I_{\rho_n}$ denotes the identity matrix of degree $\rho_n$ (the matrix presentation $A_{\overline{i}}$ of $\overline{i}$ is defined similarly to $A_{\overline{\mathrm{Nr}}}$). Therefore, due to the multiplicativity of the determinant map (see \cite[Theorem (2.11)]{AW}), we have 
\begin{align*}
 \bigl(\det A_{\overline{\mathrm{Nr}}}\bigr) \bigl( \det A_{\overline{i}}\bigr)=\det (A_{\overline{\mathrm{Nr}}}A_{\overline{i}})=\det (2I_{\rho_n})=2^{\rho_n} \neq 0. 
\end{align*}
This equality implies that $\det A_{\overline{\mathrm{Nr}}}\neq 0$, as desired. 

\begin{remark}
 For an extension of number fields $K\supsetneq F$, we set
\begin{align*}
 \mathcal{O}_{K/F}^\times=\{\alpha \in \mathcal{O}_K^\times \mid \mathrm{Nr}_{K/M}(\alpha) \in W_M \; \text{for every intermediate field } F\subset M\subsetneq K\} \; \subset \mathcal{O}_K^\times
\end{align*}
and call it the {\em group of relative units} of $K$ over $F$. The quotient group $\mathcal{E}_{K/F}:=\mathcal{O}_{K/F}^\times/W_K$ is known to be a free abelian group. In Proposition~\ref{prop:main} (or Proposition~\ref{prop:main2}), we have verified that the rank of $\mathcal{E}_{\mathbb{Q}(\zeta_n)/\mathbb{Q}(\zeta_n)^+}$ equals $0$. We remark that Odai and Suzuki have determined the rank of $\mathcal{E}_{K/F}$ for general Galois extension $K/F$ in \cite{OS1,OS2}.
\end{remark}

\bibliographystyle{amsplain}

\end{document}